\documentclass[11pt]{article} 
\usepackage{overpic}
\usepackage{bm}
\usepackage{amsmath,amssymb,latexsym}
\usepackage{amsthm} 
\usepackage{overpic}

\newtheorem{theorem}{Theorem}[section] 
\newtheorem{lemma}[theorem]{Lemma}

\theoremstyle{definition}
\newtheorem{definition}[theorem]{Definition}

\theoremstyle{remark}
\newtheorem{remark}[theorem]{Remark}

\numberwithin{equation}{section}  
\newcommand{\field}[1]{\mathbb{#1}} 
\newcommand{\R}{\field{R}}

\newcommand{\N}{\field{N}} 
\newcommand{\C}{\field{C}}
\renewcommand{\P}{\field{P}}

\newcommand{\LL}{{\mathcal L}}

\newcommand{\ZZ}{{\field{Z}}}

\renewcommand{\AA}{{\mathcal A}}
\def\RS{\mathfrak R}
\newcommand{\supp}{\mathop{\rm supp}}

\newcommand{\const}{{\rm const}}

\renewcommand{\Re}{\mathop{\rm Re}}
\renewcommand{\Im}{\mathop{\rm Im}}

\let\myo\overline
\let\myt\widetilde
\def\mdeg{\operatorname{deg}}
\title{Asymptotics of type I Hermite--Pad\'e polynomials for semiclassical functions}

\author{Andrei Mart\'{\i}nez-Finkelshtein, Evguenii A.~Rakhmanov, and Sergey P.~Suetin}

\date{This paper is dedicated to the 70th Birthday of Ed Saff}

\begin{document}

\vspace{1cm} \maketitle

\begin{abstract}
Type I Hermite--Pad\'e polynomials for a set of functions $f_0, f_1, \dots , f_s$ at infinity, 
$Q_{n,0}$, $Q_{n,1}$, \dots, $Q_{n,s}$,  is defined by the  asymptotic condition
$$ 
R_n(z):=\bigl(Q_{n,0}f_0+Q_{n,1}f_1+Q_{n,2}f_2+\dots+Q_{n,s}f_s\bigr)(z) 
=\mathcal O \left( \frac1{z^{s n+s}}\right), \quad z\to\infty,
$$
with the degree of all $Q_{n,k}\leq n$. We describe an approach for finding the asymptotic zero distribution of these polynomials as $n\to \infty$ under the assumption that all $f_j$'s are semiclassical, i.e.~their logarithmic derivatives are rational functions. In this situation $R_n$ and $Q_{n,k}f_k$ satisfy the same differential equation with polynomials coefficients.

We discuss in more detail the case when $f_k$'s are powers of the same function $f$ ($f_k=f^k$); for illustration, the simplest non trivial situation of $s=2$ and $f$ having two branch points is analyzed in depth. Under these conditions, the ratio or comparative asymptotics of these polynomials is also discussed.

From methodological considerations and in order to make the situation clearer, we start our exposition with the better known case of Pad\'e approximants (when $s=1$).
\end{abstract}

\maketitle

\section{Introduction and main results}\label{s1}

\subsection{Type I Hermite--Pad\'e polynomials}\label{s1s1}

For  $s\in \N$ let $\bm  f =(f_0, f_1, \dots, f_s )$ be a vector of analytic functions defined by their Laurent expansions at infinity,
$$
f_k(z)= \sum_{m=0}^\infty \frac{f_{m,k}}{z^m}, \quad k=0,1,\dots, s.
$$
For an arbitrary $n\in\N$ the associated vector of \emph{type I Hermite--Pad\'e (HP) polynomials},
$Q_{n,0}$, $Q_{n,1}$, \dots, $Q_{n,s}$,  corresponding to $\bm f$, 
is defined by the following asymptotic condition:
\begin{equation} 
R_n(z):=\bigl(Q_{n,0}f_0+Q_{n,1}f_1+Q_{n,2}f_2+\dots+Q_{n,s}f_s\bigr)(z) 
=\mathcal O \left( \frac1{z^{s n+s}}\right), \quad z\to\infty,
\label{1.1a}
\end{equation}
where
$$ 
 Q_{n,k} \in \P_n, \quad k=0,1,\dots, s , \quad \text{and not all } Q_{n,k}\equiv 0.
$$
We denote by $\P_n$ the space of algebraic polynomials with complex coefficients and degree $\leq n$. Function $R_n$ defined in \eqref{1.1a} is called the \emph{remainder}. 

In order to avoid unnecessary complications we will assume hereafter that all $f_k(\infty)\ne 0$; with this condition, and since the division of \eqref{1.1a} by $f_0$ preserves this asymptotic relation, we can also assume without loss of generality that $f_0 \equiv 1$. We follow this convention in the rest of the paper. 
Nontrivial polynomials $Q_{n,k}$ satisfying \eqref{1.1a}  always exist but, in general, they  are not defined uniquely. To avoid an essential degeneration we will also require that functions $f_1, \dots, f_s$ are rationally independent.  

Construction \eqref{1.1a} for the case $f_k(z) = e^{k/z}$ was introduced by Hermite in 1858, who used it to prove that the number $e$ is transcendental. The particular case $s=1$ corresponds to the (diagonal) \emph{Pad\'e approximants}, and the general construction is called Hermite--Pad\'e (HP) approximation to a collection of functions $f_0, \dots, f_s$. It plays an important role in Analysis and has significant  applications in approximation theory, number theory, mathematical physics and other fields. For details and further references see~\cite{MR2963451, MR2963452, MR1478629, MR769985, MR2796829, Sta88, MR2247778}.

A  more general classes of type I HP polynomials may be introduced by allowing arbitrary distribution of powers of polynomials $Q_{n,k}$; in this paper we keep all their degrees equal (the so-called diagonal case); neither we consider the associated construction of type II Hermite--Pad\'e polynomials (see e.g.~\cite[Ch.~23]{Ismail05}, \cite[Chapter 4]{MR1130396}, as well as \cite{ABV,ACV,MR769985,VC} for definitions). There is a well known formal relationship between the type I and type II HP polynomials   \cite{MR769985,VC} that can be elegantly expressed in terms of the associated Riemann--Hilbert problem \cite{MR2006283}. However, this algebraic connection does not mean that an explicit relation between the asymptotic behavior of both classes of polynomials exists. At least, the authors ignore any rigorously proved result establishing this correspondence for the HP zero asymptotics for a sufficiently wide class of functions. It is worth pointing out also that the zero distribution of the type II Hermite--Pad\'e polynomials is currently intensively studied by A.~I.~Aptekarev and coauthors (see e.g.~\cite{AKV}).

The asymptotic theory  of the Pad\'e approximants is fairly well developed after the works of Stahl and Gonchar--Rakhmanov in the 80ies, see e.g.~\cite{Gonchar:84, Gonchar:87, MR88d:30004a, Stahl:86}. However, the analytic theory of the HP polynomials (in particular, the asymptotics of $Q_{n,k}$ and of  the remainder $R_n$ as $n\to \infty$) is still in its infancy. In particular, we lack any general analogue of the Gonchar--Rakhmanov--Stahl (GRS) theory for HP polynomials. Some situations are rather well understood, and some particular classes of functions $\bm f$ have been studied in depth. This is the case of  the so-called Markov--type functions (that is, when the coefficients $f_{m,k}$ in \eqref{1.1a}  are moments of positive measures on the real axis), and of functions with a ``small'' set of branch points.  In this paper we focus mainly on the second class and more particularly, on functions with  a finite number of branch points, see the definition below. 

For a polynomial $Q_n\in \P_n$ we denote by $\mu(Q_n)$ its normalized zero counting measure  
$$\mu(Q_n) = \frac 1n \,\sum_{k=1}^n\,\delta_{x_k} \quad \text {where}\quad Q_{n}(x) = c_n\,\prod _{k=1}^n(1-x/x_k), \quad c_n\neq 0 $$
(if $\deg Q_n < n$ we assume that $n - \deg Q_n $ zeros of $Q_n$ are at infinity, $x_k = \infty$). 
Thus, we associate with each polynomial $Q_{n,k}$ ($n\in \N$, $k=0,1, \dots , s$) defined by \eqref{1.1a} its normalized zero counting measure
\begin{equation}
\label{def_zero_counting}
\mu_{n,k}= \mu(Q_{n,k}) 
\end{equation}
and study the weak-* convergence of the sequences $\{\mu_{n,k} \}$ as $n\to \infty$.

This problem can be reformulated in the following terms. For a (finite, Borel, and positive) measure $\sigma$ on $\C$ we denote by
$$   
C^\sigma (z) =\int \frac{d \sigma(t)}{z-t}
$$
its Cauchy transform. Recall that the weak convergence of a sequence of measures $\mu_n$ to a measure $\mu$ implies the convergence of $C^{\mu_n}$ to $C^\mu$, both in the plane Lebesque measure $m_2$ and in $L^p$, $p>1$, on compact subsets of $\C$. 

Since 
\begin{equation}
C_{n,k}(z) := C^{\mu_{n,k}}(z)= \frac1n  \frac {Q'_{n,k}(z)}{Q_{n,k}(z)}, 
\label{Cauchy1}
\end{equation}
we are equivalently interested in the behavior of $C_{n,k}$ as $n\to \infty$. Hence, the proof of convergence of these sequences and the description of their limits constitute an extension of the theorem of Stahl to the context of HP polynomials. In this generality this problem is completely open, and even a reasonable approach to its solution is not clear. The most general results so far have been obtained for the case of the type II HP polynomials for two functions ($s=2$),  each one with two branch points \cite{AKV}, and even these results are a consequence of the strong asymptotics established using the Riemann-Hilbert method, which is a clear overkill if we are only interested in the zero asymptotics. 

Thus, here we restrict our attention to a specific class of functions, general enough to be interesting, but for which we can put forward a strategy for studying the weak asymptotics. Namely, for a fixed set $\AA = \{a_1,\dots,a_p\}$  of $p\geq 2$ distinct points let
\begin{equation}
\LL_{\AA}= \left\{f(z)=f(z;\bm \alpha):=\prod_{j=1}^p(z-a_j)^{\alpha_j}:\quad
\alpha_j\in\C\setminus\ZZ,\quad
\sum_{j=1}^p\alpha_j=0 \right\}.
\label{1.1}
\end{equation}
By $\LL$ we denote the union of all classes $\LL_{\AA}$ for all finite sets $\AA$. Observe that $\LL$ is a subclass of the so-called  \emph{semiclassical} family, see the definition below. Since  each $f\in \LL$ is regular at infinity,  for convenience we always fix its branch there as $f(\infty)=1$.

With this convention the following result holds true, which shows that we can associate with the HP polynomials for a vector of functions from $\LL$ a linear differential equation with polynomial coefficients, whose fundamental system of solutions  can be build using the Hermite--Pad\'e polynomials:
\begin{theorem}~\label{main}
For any vector $\bm f = (f_0\equiv 1, f_1, \dots, f_s)$ of rationally independent functions $ f_k\in \LL$, $k=1, \dots, s$, and any $n\in \N$, each of the following $s+1$ functions 
$$   Q_{n,0},  \; Q_{n,1}f_1,   \;  \dots, \;  Q_{n,s}f_s$$ 
defined by \eqref{1.1a} satisfies a linear differential equation 
\begin{equation}
\Pi_s(z)w^{(s)}+ \Pi_{s-1}(z)w^{(s-1) }+ \dots + \Pi_1(z)w'+\Pi_0(z)w = 0,
\label{Equ1}
\end{equation}
with polynomial coefficients $\Pi_k(z) = \Pi_{k,n}(z)$ depending on $n$, whose degrees are jointly bounded  
by a value depending only on the number of branch points of the component  $f_k$ of $\bm f$.   
\end{theorem}
It follows from the definition that the reminder  $R_n$ defined by \eqref{1.1a} is also a solution of the equation \eqref{Equ1}. 

Linear  ordinary differential equations (ODE) with polynomial coefficients are one of the central topics in classical analysis, and in particular, the problem of existence of polynomial solutions for a given equation is well known. For a second order equations these polynomial solutions are called the \emph{Heine--Stieltjes (HS) polynomials} (see \cite[\S~6.8]{szego:1975}).
For relatively recent important developments in the theory of  HS polynomials see, e.g.~\cite{MR2770010,MR2647571,Shapiro2008a}, and also  \cite{Fedoryuk:1991,MR2902190, MR2377689, MR2759458, MR2897019} for the special case of the so-called Heun polynomials.

From this perspective, Theorem \ref{main} can be viewed as a construction of a class of higher order differential equations, which {\it a priori} have   polynomial solutions, as well as a fundamental system of ``quasi-polynomial''  solutions of the form $Q_{n,k}f_k$, for $f_k \in \LL$, that constitute a direct generalization of the HS polynomials.

The proof of Theorem \ref{main} is presented in Section~\ref{proofmain}. It is based on identities for the Wronskians (method which goes back to Riemann and Darboux, see \cite{MR748895} for further references) and strongly relies on the fact that any function $f\in\LL$ is \emph{semiclassical}: it satisfies the Pearson-type equation 
$$
\frac{f'}{f}(z)=\sum_{j=1}^p\frac{\alpha_j}{z-a_j},
$$
or in other words, $f$ is a solution of the first order differential equation $Af'-Bf=0$ where the roots of the polynomial $A$ are the branch points of $f$ and $\deg B \leq \deg A -2$. This method is well known for Pad\'e polynomials \cite{MR748895, MR891770}. 

For a special choice of the vector $\bm f$, which goes back again to the original ideas of Hermite, we can be more specific. Namely, let $f\in \LL$; a vector of the form
\begin{equation}
\bm f  = \left(1, f, f^2, \dots, f^s\right) 
\label{N}
\end{equation}
is a particular instance of the so-called \emph{Nikishin system} (of functions).  Condition \eqref{1.1a} defining 
the type I  Hermite--Pad\'e polynomials  for $\bm f$ now takes the form 
\begin{equation}
R_n(z):=\bigl(Q_{n,0}+Q_{n,1}f+Q_{n,2}f^2+\dots+ Q_{n,s}f^s \bigr)(z) 
=\mathcal O\left(\frac1{z^{sn+s}}\right),\quad z\to\infty,
\label{1.1N}
\end{equation}
with the assumptions that all $\deg Q_{n,k}\leq n$ and not all $Q_{n,k}$ are $\equiv 0$.

Obviously, Theorem \ref{main} applies in this case too, so we get an ODE of the form \eqref{Equ1} for the Nikishin system \eqref{N}. However, even in this situation the number of accessory parameters in the coefficients of the equation \eqref{Equ1} makes the problem virtually intractable. Thus, we  consider in more detail the case $s=2$, i.e., 
\begin{equation}
\bm f  = \left(1, f, f^2\right) .
\label{N2}
\end{equation}
Let $\AA = \{a_1,\dots,a_p\}$, $p\geq 2$, be the set of the pairwise distinct branch points of $f$, so that $f\in \mathcal L_\AA$.  
For the sake of simplicity, we will  assume initially that $\deg Q_{n,k} = n$ for all $k\in \{0,1,2 \}$,  and that the order of the zero of $R_n$ at infinity is exactly $2n+2$. In the established terminology this means that such $n\in \N$ are {\it normal indices}. If $f$ is such that every $n\in \N$ is normal, we say that  the system \eqref{N} is \emph{perfect} (see \cite{DeLoLo12,FiLo11}).

\begin{theorem}~\label{t1}
Let $\AA = \{a_1,\dots,a_p\}$, $p\geq 2$, be a set of  pairwise distinct  points,  $f\in\LL_\AA$, and $Q_{n,k}$, $k=0,1,2$, be the type I Hermite--Pad\'e polynomials for the system  $\bm f=(1,f,f^2)$, under the assumption that these functions are rationally independent.

Then functions $  Q_{n,k}f^{k}$, $k=0,1,2$, constitute three independent solutions of the differential equation 
\begin{equation}
A^2Hw'''+A\{3(A'-B)H-AH'\}w''-3(n-1)(n+2)Fw'+2n(n^2-1)Gw\equiv0,
\label{1.8}
\end{equation}
where
\begin{equation}
\label{defAB}
A(z)=\prod\limits_{j=1}^p(z-a_j), \qquad B=Af'/f\in\P_{p-2}.
\end{equation}
If $n\in \N$ if normal, then $H=H_n(z)=z^{3p-6}+\dotsb\in\P_{3p-6}$, $F=F_n(z)=z^{5p-8}+\dots \in\P_{5p-8}$, and $G=G_n(z)=z^{5p-9}+\dots \in\P_{5p-9}$ are some polynomials dependent from $n$.
\end{theorem}
It follows by definition %~\eqref{1.7} 
that the remainder  $R_n$ is also  a solution of ~\eqref{1.8}.

\begin{remark}
Although the electrostatic model for the zeros of Pad\'e polynomials (at least, in the case of the class $\mathcal L$) are well known (see Remark~\ref{remElectro} below), an analogue for HP polynomials, satisfying the third order ODE \eqref{1.8}, is currently an open problem.
\end{remark}

\medskip

Zero asymptotics of the HP polynomials is highly non-trivial (see e.g.~the results of the numerical experiments and their discussion in \cite{IKS}). In the situation of Theorems~\ref{main} and \ref{t1} its derivation is based on the substitution $w_n = \exp\left(  n \int^z v_n(t)dt  \right)$ (equivalently, $v_n = \frac 1n w'_n/w_n$), where $w_n$ is a solution of the ODE. This reduces the equation to a ``compact'' form (generalized Ricatti), for which all limit equations are purely algebraic. This argument can be carried through along convergent subsequences of functions $C_{n,k}$ (or equivalently, of weakly convergent subsequences of zero counting measures). However, for $p>2$ the existence of non-trivial set of accessory parameters  in the coefficients of the differential equation ($H$ and $G$ in the case of \eqref{1.8}) does not allow to show that the limit exists along the whole sequence $n\in \N$, without appealing to some deeper arguments. 

The situation in its full generality is so complex that we start our discussion in Section~\ref{s2} with the case of Pad\'e approximants ($s=1$). In this situation the accessory parameters still exist and constitute a problem, but the underlying Riemann surface governing the asymptotics is hyperelliptic, which makes the analysis if not simple but at least tractable. 

The second case that can be fully understood is of \eqref{N2} with $p=2$, that is, when  function $f$ has only 2 branch points. Without loss of generality they can be clearly taken to be $\pm 1$, so that we consider
\begin{eqnarray}
\label{falphan}
f(z)=f(z;\alpha)=\left(\frac{z-1}{z+1}\right)^{\alpha}, \quad 2\alpha\in\C\setminus\ZZ. %, \quad f(\infty)=1.
\end{eqnarray}
Observe that $f$ can be extended to a holomorphic function in $\C\setminus [-1,1]$ satisfying $f(\infty)=1$.

For this function we consider 
the Hermite--Pad\'e polynomials of type I  for the system  $\bm f=(1,f,f^2)$, that satisfy
\begin{equation}
R_n(z):=\bigl(Q_{n,0}+Q_{n,1}f+Q_{n,2}f^2 \bigr)(z) 
=\mathcal O\left(\frac1{z^{2n+2}}\right),\quad z\to\infty,
\label{HPfor2and2}
\end{equation}
with the assumptions that all $\deg Q_{n,k}\leq n$ and not all $Q_{n,k}$ are $\equiv 0$.

Now the statement of Theorem \ref{t1} can be made more precise:
\begin{theorem}\label{coro:casep2}
Let  $f$ be the holomorphic branch of the function \eqref{falphan} at infinity, normalized by the condition $f(\infty)=1$, and let $Q_{n,k}$, $k=0,1,2$, be the type I Hermite--Pad\'e polynomials for the system  $\bm f=(1,f,f^2)$.
Then $Q_{n,0}$, $Q_{n,1}f$ and $Q_{n,2}f^2$ (and hence, also the remainder $R_n$) satisfy the same differential equation: 
\begin{align}                                                                                                                        % (10)
(z^2-1)^2w'''
&+6(z^2-1)(z-\alpha)w''\notag\\
&-\bigl[3(n-1)(n+2)z^2+12\alpha z-(3n(n+1)+8\alpha^2-10)\bigr]w'\notag\\
&+2\bigl[n(n^2-1)z+\alpha(3n(n+1)-8)\bigr]w=0.
\label{1.9}
\end{align}
%\begin{corollary}~\label{con2}                                                                                                  
Furthermore,  polynomial $Q_{n,2}$ is a solution of the differential equation 
\begin{align}                                                                                                                           
(z^2-1)^2w'''
&+6(z^2-1)(z+\alpha)w''\notag\\
&-\bigl[3(n-1)(n+2)z^2-12\alpha z-(3n(n+1)+8\alpha^2-10)\bigr]w'\notag\\
&+2\bigl[n(n^2-1)z-\alpha(3n(n+1)-8)\bigr]w=0.
\label{1.10}
\end{align}
%\end{corollary}
\end{theorem}
\begin{remark}
Observe that the ODE \eqref{1.10} is obtained from equation~\eqref{1.9} by replacing   $\alpha$ by $-\alpha$.  
\end{remark}

The absence of accessory parameters in the coefficients of the ODE in  Theorem~\ref{coro:casep2} allows us to obtain an explicit formula for the limit zero distribution of the HP polynomials. In order to formulate the assertion, it is convenient to introduce a different branch of the function $f$ from \eqref{falphan}; namely, let in what follows
\begin{eqnarray}
\label{falphNalt}
 f_0(z)=\left(\frac{1-z}{1+z}\right)^{\alpha}, \quad z\in \C\setminus F, \quad F=\overline \R\setminus (-1,1),
\end{eqnarray}
denote the holomorphic branch in $\C\setminus F$, fixed by  $ f_0(0)=1$.  
Function
$$
\rho_n(z)=Q_{n,1}(z) +2\cos(\alpha\pi) Q_{n,2}(z) f_0(z),
$$
defined and holomorphic in $\C\setminus F$, will play a prominent role in the convergence of HP approximants (see Theorem~\ref{t4} below). Meanwhile, observe that
$$
\rho_n(x)=Q_{n,1}(x) +2\cos(\alpha\pi) Q_{n,2}(x) \left( f^++f ^-\right)(x) , \quad x\in (-1,1),
$$
where $f^+$ (resp., $f^-$) are the boundary value of the function \eqref{falphan} on $(-1,1)$
from the upper (resp., lower) halfplane. 

We have:
\begin{theorem}~\label{t2}                         
Under the assumptions of Theorem~\ref{coro:casep2}, let $\alpha\in \R$, $|\alpha|\in(0,1/2)$. Then all the zeros of polynomials $Q_{n,k}$ belong to
$\R\setminus[-1,1]$, while $\rho_n$ has at least $2n+1$ zeros on $(-1,1)$.

Moreover, there exist two unit measures on $\R$, $\lambda$ supported on $E=[-1,1]$, and $\nu$ supported on $F=\overline \R\setminus (-1,1)$, such that the 
normalized zero counting measures $\mu_{n,k}$, $k=0, 1, 2$, defined in \eqref{def_zero_counting}, converge to $\nu$, while the normalized zero counting measure of $\rho_n$,
$$
\frac{1}{2n} \mu(\rho_n)=\frac{1}{2n}  \sum_{\rho_n(x)=0}\delta_x
$$
(with zeros accounted according to their multiplicity) converges to $\lambda$. 

Measures  $\nu$ and $\lambda$ are absolutely continuous,  with the densities
\begin{align}                                                                                                              
\nu'(x)& =\frac{\sqrt{3}}{2\pi}\,\frac{1}{\sqrt[3]{x^2-1}}\,\left(
\frac{1}{\sqrt[3]{|x|-1}}-\frac{1}{\sqrt[3]{|x|+1}}\right),\quad x \in  \overline \R\setminus (-1,1).
\label{1.91} 
\\
\label{er2}
\lambda'(x)&= 
\frac{\sqrt{3}}{4\pi}\,\frac{1}{\sqrt[3]{1-x^2}}\,\left(
\frac{1}{\sqrt[3]{1-x}}+ \frac{1}{\sqrt[3]{1+x}}\right), \quad x\in (-1,1).
\end{align}
\end{theorem}
\begin{remark}
If $\alpha\in \C$, $|\alpha|\in(0,1/2)$, then not necessarily all zeros of $Q_{n,k}$ are real; however, the asymptotics \eqref{1.91}--\eqref{er2} remains valid.
\end{remark}

It follows from the results and methods of  \cite{MR3058747,MR3137137} that the  limit zero distributions  $\lambda$ and $\nu$  above %of $\mu_{n,k}$ %Theorem~\ref{t2} 
can be characterized by an equilibrium problem involving mixed potentials, as follows. 
For a measure $\mu$ on $\C$ we denote by
$$
V^\mu(z)=\int \log\frac{1}{|z-t|}d\mu(t)
$$
its logarithmic potential. Furthermore,  for $F=\R\setminus (-1,1)$, let $g_F(z,t)$ be the  Green function for the domain $\Omega=\C\setminus{F}$ with pole at $t$; for a measure $\mu$ supported in $\Omega$,
$$
G^\mu_F(z)=\int_E g_F(z,t)\,d\mu(t)
$$
defines its Green potential (with respect to $\Omega$).
 
Then $\lambda$ in Theorem~\ref{t2} is the unique probability equilibrium measure supported on the interval $E=[-1,1]$ for the mixed Green-logarithmic potential,  characterized by the following identity:
\begin{equation}
3V^\lambda(x)+G^\lambda_F(x)\equiv\gamma_E=\const,\qquad x\in E.
\label{eq1}
\end{equation}
Furthermore,  $\nu$ in Theorem~\ref{t2} is the balayage of $\lambda$ from  $\Omega$ onto $F$ (see e.g.~\cite{Landkof:72} or \cite{Saff:97} for the definition of balayage). 

Equivalently, $\nu$ is the probability equilibrium measure for the mixed Green--logarithmic potential with respect to $E$, but now in the external field 
$$
\psi(x)=3g_E(x,\infty) = 3 \log\left(|x| + \sqrt{x^2-1}  \right), \quad x\in F,
$$
that is,  
\begin{equation}
3V^\nu(x)+G^\nu_E(x)+\psi(x)\equiv\gamma_F=\const,\qquad x\in F,
\label{eq2}
\end{equation} 
where again
$$
G^\nu_E(z)=\int_F g_E(z,t)\,d\nu(t),
$$
and $g_E(z,t)$ is the Green function of $\myo\C\setminus{E}$ with pole at $t$. The equivalence of such kind of problems (in a more general setting) was discussed in~\cite{BS2015}. It was also explained there that these problems are not a consequence of the vector equilibrium problems studied in \cite{Gonchar:81,Gonchar:85,Nik86} for two types of systems of Markov functions.

Equilibrium conditions above were used in ~\cite{MR3058747, MR3137137} to investigate the convergence of Hermite--Pad\'e approximants for a set of functions $1,f_1,f_2$ under the assumptions that the  pair $f_1,f_2$   forms  a complex Nikishin system for a so-called
 {\it Nuttall condenser} $(E,F)$, where the plate $E$ is a finite union of real intervals and the second plate $F$ exhibits a symmetry known as the $S$-property. 
In ~\cite{MR3137137}  (see also \cite{MR3088082}) it was shown that the Cauchy transform $h=C^\nu$ is a solution of a cubic equation, so that the Nuttall condenser can be associated with a three sheeted Riemann surface (see also \cite{KovSu14}).   

Finally, the problem of the \emph{strong asymptotics} of the HP polynomials 
seems to be open even in the simplest cases. For Pad\'e approximants ($s=1$) such formulas can be fairly easily obtained from the differential equation (see Section~\ref{s2}) using the Liouville--Green approximation \cite{Olver74b}. 
However, the Liouville--Green (a.k.a.\ Liouville-Steklov or WKB)  method is not yet completely developed for higher order differential equations  (for the case of two-point Pad\'e approximants see~\cite{KoSu14}), so it is not clear how to obtain the strong asymptotics of $Q_{n,2}$ using WKB even in the simplest case of \eqref{1.10} with all its coefficients explicit.

In the case of \eqref{falphan}--\eqref{HPfor2and2}, since for $|\alpha|\in(0,1/2)$  
the singularities of the function  \eqref{falphan} at $\pm1$ are integrable,    
polynomials $Q_{n,2}$ satify some {\it (bi)orthogonality}  relations on
$\myo\R\setminus(-1,1)$ with respect to the weight
$\bigl((x-1)/(x+1)\bigr)^{\alpha}$. It allows us to establish convergence of the ratio of HP polynomials, which as it follows from Theorem~\ref{t2}, will occur only on the complex plane cut along $F=\overline \R\setminus (-1,1)$. 
\begin{theorem}\label{t4}
Under the assumptions of Theorem~\ref{t2},  as $n\to\infty$, 
\begin{gather} 
\frac{Q_{n,1}}{Q_{n,2}}(z)\longrightarrow
-2\cos(\alpha\pi) \,  f_0(z) ,
\label{4.4}\\
\frac{Q_{n,0}}{Q_{n,2}}(z)\longrightarrow
f_0^{\,2}(z), 
\label{4.5}
\end{gather}
locally uniformly in  $\C\setminus F$.
\end{theorem}
\begin{remark}\label{rem1}
Similar results for Markov functions have been proved recently in \cite{LM15b,LM15a,LM15c}.

For convenience, Theorems~\ref{t2} and \ref{t4} assumed that $|\alpha|\in(0,1/2)$, although the crucial constraint is that neither $\alpha$ nor $2\alpha$ are integers. Indeed, if $\alpha, 2\alpha\in\R\setminus\ZZ$, but $|\alpha|\not \in(0,1/2)$, then function $f_0^2$ (see the definition in~\eqref{falphNalt}) is no longer integrable on the interval $E$. By multiplying the defining relation \eqref{HPfor2and2} by a suitable real polynomial of a fixed degree, say $m$, we can reduce the problem to the previously analyzed situation of an integrable $f_0$, at the cost of loosing $m$ orthogonality conditions satisfied by $Q_{n,k}$ and $R_n$ (see Section~\ref{s4}). Although the arguments that lead to the proof of the limit zero distribution in Theorem~\ref{t4} still apply in this case (so \eqref{1.91} and \eqref{er2} are valid as long as $\alpha, 2\alpha\in\R\setminus\ZZ$), a finite number of the zeros of the polynomials $Q_{n,k}$ is now out of our control. In consequence,  instead of a uniform convergence as in Theorem~\ref{t4} we can claim~\eqref{4.4} and~\eqref{4.5} in capacity only.

We can extend these conclusions even further, to the case when $\alpha, 2\alpha\in\C\setminus\ZZ$; this is subject of a manuscript in preparation.
\end{remark}

The proof of Theorem~\ref{t4}, carried out in Section~\ref{s6}, relies on the existence of a Riemann surface with the so-called Nuttall's canonical partition (see \cite{MR769985}). As a consequence, we conclude that for the remainder $R_n$, defined by \eqref{HPfor2and2}, there exists the limit  $g(z) = \lim (1/n)\log|R_n(z)|$, which is a single-valued harmonic function on this Riemann surface with the prescribed behavior at infinity on each sheet of $\RS$ (what we call the \emph{real $g$-function} of the Riemann surface $\RS$):
\begin{definition}\label{def:gfunction}
Let $\RS$ be a $(s+1)$-sheeted compact Riemann surface, and $ \pi :\RS \to\overline\C$,  $\pi(\zeta) = z$, $\zeta \in \RS$, its canonical projection. For $z\in \C$ we use the notation $  \pi^{-1}(z) = \{\zeta^{(0)}, \zeta^{(1)}, \dots, \zeta^{(s)}\}$. The  \emph{real $g$-function} $g = g_{\RS}$ on $\RS$ is defined as a function harmonic on the finite part of $\RS$ and satisfying conditions
\begin{align*} 
g(\zeta) = & -s\log|z| + \mathcal O(1), \quad \zeta \to \infty^{(0)}, \\ 
g(\zeta) = & \log|z| + \mathcal O(1), \quad \zeta \to \infty^{(j)}, \quad j=1, \dots, s, 
\\
g(\zeta^{(0)})& +   g(\zeta^{(1)}) +\dots + g(\zeta^{(s)})  = 0,
\end{align*}
with 
$$  \pi^{-1}(\infty) = \{\infty^{(1)}, \dots, \infty^{(s+1)} \}.$$ 
Such a function exists and is unique. 

Furthermore, any function $G$ on $\RS$ such that $\Re G=g$ is the \emph{complex $g$-function} of $\RS$, and can be expressed as an Abelian integral of the third kind.
\end{definition}

In theory, the considerations above can be extended to the most general setting of HP approximants defined by \eqref{1.1N} or even \eqref{1.1a}, with $f_k\in \mathcal L$, by saying that the limit  of any convergent subsequence of $  (1/n)\log|R_n(z)|$ is the real $g$-function corresponding to certain 
$(s+1)$-sheeted  Riemann surface $\RS_{s+1}$ associated with $\bm f$. 
Moreover, all limits of all the Cauchy transforms \eqref{Cauchy1} along the same subsequence are rational functions on \emph{the same} Riemann surface $\RS_{s+1}$.

Finally, if $\RS_{s+1}$ can be uniquely prescribed by some additional conditions we conclude that the whole sequence  $(1/n)\log|R_n(z)|$ converges. Such a condition could be formulated in terms of the zero-level set $\{ g(z)=0\}$ and branch types of functions in $\bm f$. 

A complete implementation of the plan outlined above in the most general setting is rather a program for a future. Certain assertions needed for formal justification of this program are not completely  proved at the moment, and volume restrictions prevent us from presenting even the main details related to general case in this paper.  Instead, we discuss next a version of our approach for the case $s=1$, that is, for the Pad\'e approximants, situation still far from being trivial. As it was mentioned, the Riemann surfaces related to this case are hyperelliptic and this is an essential simplification. Also, all facts related to our method for $s=1$ can be obtained as corollaries of the Stahl's theorem. However,  we intend to extend the method to the study of the Hermite--Pad\'e polynomials where we lack any analogue of the Stahl's theorem. It turns out that without recurring to this theorem even the case of Pad\'e polynomials presents challenges.

\begin{remark}
These results were in part announced in~\cite{MR3088082}. We wish to acknowledge  the useful remarks of A.~B.~J.~Kuijlaars on the first version of this manuscript, as well as of the anonymous referees.
\end{remark}

\section{Pad\'e  polynomials}\label{s2}

We start the discussion illustrating our approach in the simplest situation, when $s=1$ and $\bm f= (1,f)$ in \eqref{1.1a}. This corresponds to the 
\emph{Pad\'e approximants} to the function $f$, or rather, to its analytic germ at infinity; recall that without loss of generality we assume that $f(\infty)=1$, so that
 \begin{equation}
 \label{germ}
 f(z)= 1+ \sum_{m=1}^\infty \frac{c_{m}}{z^m}.
 \end{equation}
It is customary to use for the polynomials defined in \eqref{1.1a}  the notation $P_n= -Q_{0,n}$ and $Q_n=Q_{1,n}$, so that we have the conditions $ P_n,  Q_n\in \P_n$, $Q _n\not\equiv0$, and
\begin{equation}
R_n(z)=\bigl(Q_n f -P_n\bigr)(z)=\frac{M_n}{z^{n+1}}\left( 1 + \mathcal O\left(\frac1{z}\right)\right),\quad z\to\infty.
\label{Pade}
\end{equation}
Polynomials $P_n$ and $Q_n$ (not uniquely defined  by \eqref{Pade}), are called the \emph{Pad\'e polynomials} ($Q_n$ are Pad\'e \emph{denominators}), and the rational function $ P_n/Q_n$ (which is uniquely defined) is the \emph{Pad\'e approximant} to $f$ of order $n$. Recall that an index $n$ is called \emph{normal} if $\deg Q_n = n$ and $M_n\neq 0$ in \eqref{Pade}.  

Pad\'e approximants constitute a classical method of rational approximation of analytic functions. They are the best local approximations to a power series or, alternatively, they are convergents of a continued fraction (Jacobi or Chebyshev fraction) to this series. A systematic study of such fractions was started in the 18th century by Euler, Lagrange and others, although the ground of the analytic theory was laid in the 19th century in the works of Hermite, Chebyshev and their students and followers, such as Pad\'e, Markov, Stieltjes and others (see e.g.~\cite{Baker96} and a recent review \cite{MR2963451}).

\subsection{The differential equation} \label{sec21}

As before, we particularize our analysis to the case of a function $f\in \LL_\AA$, where $\AA = \{a_1,\dots,a_p\}$, $p\geq 2$, and $a_j$ are pairwise distinct. 
The following theorem belongs basically to Laguerre \cite{MR1503827}, although its derivation, also classical (see Remark~\ref{rem3} below), is different from the original one. We include a more or less detailed proof here mainly with illustrative purposes, having  in mind its extension to HP polynomials in Section~\ref{sec:ODE}:
\begin{theorem}[see  \cite{MR3058747}]\label{t3}               
Let function $f(z) \in \LL_\AA $, $\AA = \{a_1,\dots,a_p\}$, $p\geq 2$, be given by~\eqref{germ},  with polynomials $A$ and $B$  defined by \eqref{defAB}. If 
$P_n$, $ Q_n$ are the associated Pad\'e polynomials of degree $n$, and $R_n$ the remainder~\eqref{Pade}, then there exist polynomials $H_n$ and $C_n$, and a constant $N$ (in general, all depending on $n$) such that   functions $P_n$,   $Q_n f$ and  $R_n$ are solutions of the differential equation with polynomial coefficients
\begin{equation} \label{MainPade}     
AH_nw''+\{(A'-B)H_n-AH_n'\}w' - N C_n w=0,  % (28)       
\end{equation}
and $\deg (C_n) - \deg (H_n) \leq p -2 $.  

Furthermore, if the index $n$ is normal then 
$$
N=n(n+1), \quad  H=H_n(z)=z^{p-2}+\dotsb \in\P_{p-2}, \quad C=C_n(z)=z^{2p-4}+\dotsb \in \P_{2p-4}.
$$  
\end{theorem}
\begin{remark}~\label{rem3}
In  connection with this result see~\cite{Chud,MR748895,MR2964145,MR891770} and 
\cite[\S~3, Theorem~3.1, formula~(3.1)]{MR3058747}). According to J.~Nuttall \cite{MR891770},  says that the method of proof can be traced all the way back to Riemann~\cite{Riemann}.

Since $f(\infty)=1$, substitution of  $\alpha_j$ for $-\alpha_j$ is equivalent to dividing of both sides of~\eqref{Pade} by $f$. It follows that polynomial $Q_{n}$ is a solution of the equation obtained from~\eqref{MainPade} by changing sign of all $\alpha_j$, which basically replaces $B$ by $-B$.
\end{remark}
\begin{remark}~\label{remElectro}
In terminology of \cite{MR2770010}, $P_n$'s are Heine--Stieltjes polynomials.  It is well known (see \cite{MMM,MR2770010}) that \eqref{MainPade} yields an electrostatic model for the zeros of $P_n$'s: they are in equilibrium in the external field created by the masses fixed at the zeros of $A$ and $B$, plus by a number of ``ghost'' or moving charges, corresponding to the zeros of $H_n$.
\end{remark}

\begin{proof}
We fix a neighbourhood $U=\{z:\, |z|>R\}$ of $z=\infty$ where series \eqref{germ} is convergent.  
If  $w(z)=c_1(Q_{n}f)(z) + c_2P_{n}(z)$ is any linear combination of functions $(Q_{n}f)(z)$ and $P_{n}(z)$ then by the general theory, the associated Wronskian vanishes identically,
\begin{equation} 
W[w,P_n, Q_nf](z)  \equiv                                                                                                 
\begin{vmatrix} w&w'&w''\\
P_{n}&P_{n}'&P_{n}''\\
Q_{n}f&(Q_{n}f)'&(Q_{n}f)''
\end{vmatrix}\equiv0,\quad z\in U.
\label{2.1}
\end{equation}
Since $W(z)$ is an analytic function, it must vanish identically in its whole domain of analyticity, $\C\setminus \AA $. Expanding the determinant along the first row yields  the following second order differential equation with respect to $w$, solved in particular by $w = P=P_n$, $w=Qf= Q_nf$ and in consequence, by $w =R= R_n$:
\begin{equation}
W(z) =  W_2(z) w''(z) + W_1(z) w'(z) + W_0(z) w (z) = 0 ,
\label{ode2.2}
\end{equation}
where (we omit the explicit reference to $n$)
\begin{align}
W_2(z) &=  
\begin{vmatrix} 
P&P'\\
Qf&(Qf)'
\end{vmatrix}
=
\begin{vmatrix} 
P&P'\\
R&R'
\end{vmatrix}=
\mathcal O \left(\frac 1{z^2}\right)\quad\text{as}\quad z\to \infty,
\label{2.2}
\\
 -W_1(z) & =  
\begin{vmatrix} 
P&P''\\
Qf&(Qf)''
\end{vmatrix}
=
\begin{vmatrix} 
P&P''\\
R&R''
\end{vmatrix}=
\mathcal O \left(\frac 1{z^3}\right) \quad\text{as}\quad z\to \infty,
\label{2.3}
\\
W_0(z) & =  
\begin{vmatrix} 
P'&P''\\
(Qf)'&(Qf)''
\end{vmatrix}
=
\begin{vmatrix} 
P'&P''\\
R'&R''
\end{vmatrix}=
\mathcal O \left(\frac 1{z^4}\right) \quad\text{as}\quad z\to \infty.
\label{2.4}
\end{align}
The second equality in each  \eqref{2.2}--\eqref{2.4} is obtained by subtracting the first row of the determinant from the second one, and the order of zero for $W_j$ at infinity is found using \eqref{Pade}. % the fact that this order for the remainder $R$ is $n+1$; for the first and second derivatives it is respectively  $n+2$ and $n+3$.

So far, \eqref{ode2.2}--\eqref{2.4} is just a set of straightforward identities. But now we use the semiclassical character of $f$:  we have $Af'=Bf$ where $A = z^p + \dots,$ $B\in\P_{p-2}$ and further 
$$ %\begin{equation}
A^2f''=(B^2+B'A-A'B)f=B_1f,\quad\text{where}\quad B_1\in\P_{2p-3}.
$$
%bel{} \end{equation}

Taking this into account we multiply equation \eqref{2.1} by  $A^2/f$ reducing it to an equation with 
polynomial coefficients
\begin{equation}
\frac{A_p^2(z)}{f(z)}W(z) %(w,P_0,P_1f)
=\Pi_2(z)w''+\Pi_1(z)w'+\Pi_0(z)w\equiv0,
\quad z\in\C\setminus \AA.
\label{2.5}
\end{equation}
Consider for instance the coefficient $\Pi_2$; from \eqref{2.2} we obtain
\begin{equation}
\Pi_2(z) =\frac{A^2}{f} W_2(z) = 
A
\begin{vmatrix} 
P&P'\\
AQ&A(Qf)'/f
\end{vmatrix} (z)= A \begin{vmatrix} 
P&P'\\
AQ&AQ'+BQ
\end{vmatrix} (z)= A(z) H(z),
\label{2.6}
\end{equation}
where $H$ is a polynomial; by the right hand side of \eqref{2.2}, $H\in\P_{p-2}$.  In a similar vein it is established that
$$
\Pi_j = \left(\frac{A^2}{f}\right) W_j \in \P_{2p-4+j}, \quad j=0, 1, 2.
$$ 
Furthermore, an easy consequence of \eqref{2.2} and \eqref{2.3} is that $W'_2 = - W_1$. From here,
\begin{equation}                                                                                                                   
\Pi_1(z)=-\biggl(\frac{f}{A^2}\Pi_2\biggr)'\frac{A^2}{f}=(A'-B)H-AH'.
\label{2.7}
\end{equation}

Finally, for a normal index $n$, \eqref{2.5} has a polynomial solution $P$ of degree exactly $n$. Substituting $P$ into the equation one can calculate the coefficient of the highest power of $z$ (which is $z^{n+2p - 4 }$) and find that the leading coefficient of $\Pi_0$ is $N = - n(n+1)$.

Theorem ~\ref{t3} is proved.
\end{proof}

%\subhead2.2\endsubhead
\begin{remark}~\label{rem2}
Arguments presented in the proof of theorem above show that if we do not assume normality, then   condition \eqref{Pade}  should be replaced by
\begin{equation}
\label{Pademod}
R_n(z)= \frac{M_n}{z^{n+1+\ell_n}} \left(1 + \mathcal O\left(\frac1{z}\right)\right),\quad z\to\infty,
\end{equation}
with $0\leq\ell_n\leq{p-2}$, so that all indices $n\in\N$ are almost perfect, according to the terminology of~\cite{MR748895}. Indeed, if in \eqref{Pademod} we assume  $\ell_n\geq p-1$ then by the arguments above  we would obtain that in \eqref{MainPade}, $\deg{H}<0$, that is $h\equiv0$ and  $\Pi_1=\Pi_0=0$, which is  not possible. 
It follows that for $n' = \mdeg P_n = \mdeg Q_n$ we have $n' \geq  n-\ell_n$ and further $N = n'(n'+1)\geq n(n+1) - 2n\ell_n. $

\end{remark}

Let us finally single out the simplest non-trivial case of Theorem~\ref{t3},  when $f\in \LL_\AA$, with $\AA = \{a_1,a_2, a_3\}$. Then the coefficients for the normal indices in the differential equation~\eqref{MainPade} have the form $H_n(z) = z-z_n$, $C_n(z) = (z-b_n)(z-v_n)$, where   $z_n$, $b_n$, $v_n$ are some parameters, so that   equation~\eqref{MainPade} becomes
\begin{equation}
\label{odeNutall}
A(z)(z-z_n)w''+\bigl((A'-B)(z)(z-z_n)-A(z)\bigr)w'
-n(n+1)(z-b_n)(z-c_n)w=0.
\end{equation}
This was the equation  used by J.~Nuttall~\cite{MR891770} to derive a formula of strong asymptotics for Pad\'e polynomials and their remainder in the case $f\in \mathcal L_{\mathcal A}$, and which we discuss next in the general setting.

\subsection{Asymptotics}\label{s2s1}

The zero asymptotics of numerators and denominators of Pad\'e approximants for functions with singularities constituting a set of zero capacity  was derived in a seminal work of H.~Stahl \cite{MR88d:30004a,Stahl:86}, see the details below. However, the proof of Stahl's theorem is based on rather sophisticated potential theoretical arguments, and has no simplification even for the semiclassical functions from $\mathcal L$, for which we have the bonus of the Laguerre differential equation  \eqref{MainPade}. 
It is tempting to use \eqref{MainPade} to study the asymptotic behavior of the Pad\'e polynomials, finding a simpler proof of Stahl's result, at least 
for the class $\mathcal L$. 

The first attempt in this direction for non-classical situation, although not successful, is due to Laguerre himself \cite{Laguerre}. 
A hundred years later J.~Nuttall~\cite{MR891770} repeated the attempt, this time successfully, even before the publication of Stahl's results. Since the Laguerre differential equation is of order 2, it is natural to use the Liouville--Green (Liouville--Steklov or WKB) method, which actually yields even the strong asymptotics of the polynomials (at least, away from the support of the limiting zero counting measure). However, in the case of functions $f\in \mathcal L$ with $p>2$, equation  \eqref{MainPade} has non-trivial accessory parameters $H_n$ and $C_n$ whose asymptotics is  a priori not clear. Their behavior was conjectured by Nuttall first, and then rigorously proved for $p=3$ 
in the pioneering work~\cite{MR891770}, where the crucial step relied on the Stahl's theorem just appeared in \cite{Stahl:86}! Nuttall's method was extended to the case of an arbitrary number of branch points in \cite{MR2964145}. Again, asymptotics of the free parameters along full sequence $n$ is  obtained using Stahl's theorem.

In the approach above the goal was the strong asymptotics, and the Olver's lemma \cite{Olver74b} played the key role. Now, interested in the weak asymptotics, we proceed slightly differently: we use the Ricatti substitution in \eqref{MainPade} and take limits along converging subsequences, resulting in an algebraic, instead of a differential, equation.
%standard WKB approach we rewrite \eqref{MainPade} in the Riccati form and take limits along converging subsequences.   

Namely, 
substitution 
\begin{equation} \label{WKB}
w_n = \exp\left( n \int^z v_n(t)dt  \right) \quad \text {or, equivalently, }\quad v_n = \frac 1n w'_n/w_n
\end{equation}
in equation \eqref{MainPade} and division by $AH_n$ reduces it to the Ricatti equation
\begin{equation}\label{Ric}
- \frac 1n v'_n = v_n^2 + s_n v_n + r_n,
\end{equation}
where 
\begin{equation} \label{Ric-1}
s_n =  \frac 1n \left(\frac {A'}{A} - \frac {B}{A} - \frac {H_n'}{H_n}\right) v_n 
\qquad \text {and }\qquad 
r_n = -\frac {N}{n^2}\frac {C_n}{AH_n}.
\end{equation}
Equation \eqref{Ric} has, in particular, solutions $v_n = P'_n/(nP_n)$, $Q'_n/(nQ_n) + f'/(nf)$ and $R'_n/nR_n$.

Due to the weak compactness of the zero-counting measures involved in the coefficients of the equation \eqref{Ric}, we have that
$$
\lim_n \frac{v_n(z)}{n}= \lim_n s_n(z) =0 \quad m_2 \text{-a.e.}
$$
on compact subsets of $\C$. Furthermore, $N/n^2 \to 1$ as $n\to \infty$. Consequently, \begin{lemma}\label{L2} 
We have 
$$
\lim_n \left( v^2_n - \frac{C_n}{AH_n}\right) = 0 \quad m_2 \text{-a.e.} $$
on compact subsets of $\C$.
\end{lemma}
In particular, this procedure will yield an algebraic (vs.~a differential), and more precisely, a quadratic asymptotic expression.

The sequence of rational functions ${C_n}/{AH_n}$ is compact in the $m_2$ convergence, with possible limits including identically $\infty$ and $0$; this  happens if some zeros of $C_n$ or $H_n$ (or both) go to infinity. To handle such an event one has to use the spherical normalization of polynomials,  and  normalize the term $ v^2_n$ accordingly. 

Let us select convergent subsequences 
$$C_n \to C \qquad \text {and }\qquad  H_n \to H \qquad \text {as }\quad n\in \Lambda\subset \N. $$ 
In the generic case we have $\mdeg H = p-2,\, \mdeg C = 2p-4$; otherwise, the degrees of $C$ and $H$ may be reduced. Eventually, this is not important since we have cancellation of possible large zeros of $C_n$ and $H_n$. More exactly, 
\begin{lemma}\label{L3} We have $C = VH$ where  $V(z) = z^{p-2} +\dots$.\end{lemma}
This lemma is eventually a  consequence of the fact that the only possible singularities of any solution of \eqref{MainPade} are the zeros of $A$.

In \cite{MR3058747} one can find the proof of a stronger assertion: for any large enough $n$ we have representation $C_n = \tilde H_n V_n $ where $|\tilde H_n - H_n| = \mathcal O(1/n)$ on compacts in $\C$ provided that polynomials are spherically normalized. 

However, as it was mentioned before, all these arguments work along subsequences of $\N$, and the uniqueness of the limiting polynomial $V$ (and thus, of the asymptotic zero distribution) does not follow from this analysis. For $p>2$ (see \cite{MR2964145,MR891770}), uniqueness has been established so far recurring to the Stahl's theorem on the weak asymptotics, that we briefly outline next\footnote{A reader interested strictly in the proof of the main results (Theorems~\ref{main}--\ref{t4}) may skip the rest of this section and move to Section \ref{sec:ODE}.}. It is worth mentioning however a new approach, discussed in~\cite{RakhPreprint}, which uses fixed points arguments to prove uniqueness of $V$. 

Let us recall that for a function from the class $\LL$ Stahl's theorem asserts that there exist a unique (up to subsets of capacity zero) compact  set $F = F_f \subset \C$, which is a union of analytic arcs, with the following properties: the complement to $F$ is connected, $f$ is single-valed in $\overline\C\setminus F$, the jump of $f$  across any arc in $F$ is not identically zero and, finally, the \emph{$S$-property} holds:
$$
 \frac{\partial g}{\partial n_+}(z) =  \frac{\partial g}{\partial n_-}(z) ,  \quad z\in F^0,
$$
where   $g$ denotes the Green function  of $\C\setminus F$ with pole at infinity,  and $n_\pm$ are two oppositely directed normals to $F^0$, where $F^0$ is the union of open parts of arcs constituting $F$. Furthermore, the Robin (equilibrium) measure of such a compact  $F$ is precisely the weak-* limit of the zero counting measures $\mu(Q_n)$ for the  Pad\'e deminators $Q_n$. It was proved also that  the sequence of Pad\'e approximants $\pi_n = P_n/Q_n$ associated with $f$ converges to the function $f$ in capacity in the complement to $F$; the exact rate of convergence in capacity was also determined.

The proof of the existence of such a set $F_f$ relied on its characterization as the set of minimal capacity:
\begin{equation*}
%\label{MinCap}
\text{cap}(F_f) =\min \left\{ \text{cap} (F):\,   f \text{ is holomorphic and single--valued in } \C\setminus F \right\}.
\end{equation*}

This extremal problem is close to the classical Chebotarev's problem of minimal capacity in the class  of all continua $F$ on plane containing $\AA$. For finite sets $\AA$ it was solved  by Grotsch's in 1930 \cite{Groetsch1930}. Stahl's result is more general even for finite sets $\AA$; it is actually a theorem from the geometric function theory related to a version of a general moduli problem; see \cite{Strebel:84}.  For finite sets $\AA$ a simple solution of the existence problem based on max-min energy problem was given in \cite{Rakhmanov94}; see also \cite{MR2770010} for a study of so called {\it critical measures} in plane which present  another generalized version of moduli problem. Stahl's theorem was extended later to the case of the existence of an external field (or to a varying orthogonality) by Gonchar and Rakhmanov in \cite{Gonchar:87}, and to more complex equilibria in \cite{Buslaev2012,Sproperty2011}. It is worth mentioning also \cite{Bus13} where an analogue of Stahl's
theorem for the case of $m$-point Pad\'e approximants was proved. In this situation the
external field is given by unit negative charges supported at the $m$
interpolation points, and as a consequence in a ``generic case'' the corresponding
 $S$-curve makes an optimal partition of the Riemann sphere into $m$ domains
centered at these $m$ interpolation points.

The proof of the convergence assertions of Stahl's theorem   is based on the complex (non-hermittian) orthogonality conditions for Pad\'e denominators $Q_n$, of the form
\begin{equation*} %\label{Ort}  
\oint_F Q_n(z) z^k f(z) dz = 0, \quad k =0,1, \dots, n-1,
\end{equation*}    
where integration is taken over any system of contours separating $F$ from infinity.

We can describe the extremal set $F_f$ in terms of trajectories of a quadratic differential:  there  exists a polynomial $V(z)=V_f(z) = \prod_{i=1}^{p-2} (z-v_i)$ such that the quadratic differential  $-(V/A)\,(dz)^2$ on the extended plane $\overline \C$ is closed: all its trajectories, given by
$$ 
\frac { V(z)}{A(z)} (dz)^2 < 0  ,
$$
are either closed contours or critical arcs, joining poles or zeros of $AV$.

Function $\sqrt{ {V(t)}/{A(t)}}$ has a holomorphic branch in $\Omega = \overline\C\setminus F$   and the  Green function for $\Omega $ with pole at infinity can be written as
\begin{equation} \label{Gre}  
g(z) = \Re G(z),\qquad  G(z) = \int_a^z \, \sqrt{ {V(t)}/{A(t)}}\, dt, \quad (a\in \mathcal A),
\end{equation}
where branch of the root is selected by the condition 
$$
\lim_{z\to \infty} z \sqrt{ {V(t)}/{A(t)}} = 1.
$$

Formula \eqref{Gre} shows that function $g$  has a harmonic continuation to   the hyperelliptic Riemann surface $\RS$  of the function $\sqrt{ VA}$,   which is convenient to  interpret as a two sheeted branched covering over $\overline\C$. In fact, it is the real $g$-function of $\RS$, while $G$ is the corresponding complex $g$ - function, in the sense of Definition~\ref{def:gfunction}.

Now, the problem is to prove that the Riemann surface $\RS$ is uniquely determined by the function $f.$ We suggest a procedure consisting of two steps. The first step is the determination of the family of all Riemann
surfaces associated with all functions $f \in \mathcal L_{\mathcal A}$ with a fixed set $\mathcal A$ of branch points. This determination is made by using special properties of the $g$-functions associated with Riemann surfaces $\RS$ originated by compacta $F_f$. 

It follows from Stahl's theorem that the family of possible Riemann surfaces $\RS$ is finite. Independently of Stahl's theorem we can assert this fact from the following perspective (in the situation of common position). 

Consider all hyperelliptic Riemann surfaces with $2p-2$ quadratic branch points out of which $p$ are fixed at zeros of $\AA$ and the remaining $p-2$, not necessarily distinct and that we denote by $v_1, \dots, v_{p-2}$, are free; let's
$$
V(z)=\prod_{i=1}^{p-2} (z-v_i).
$$  
It turns out that the requirement that the derivative $G'$ of the complex $g$-function can have poles only at $\AA$ singles out only   a finite number of such polynomials $V$.   Note that the formula $G' = \sqrt{V/A}$ above tells us that this must hold indeed for $g$-functions originated by Stahl's compacta $F_f$. We assert that this property is characteristic for the $S$-compacta for functions with branch points at the roots of $A$.
   
Next, it is clear that the extremal compact $F_f$ is the projection of the zero level of $g_{\RS}$, or, better to say, the zero level of $g_{\RS}$ is the lifting of $F_f$ onto $\RS$, that is, 
$$ F^*  = \left\{\zeta \in \RS:\quad g(\zeta) = 0 \right\} = \pi^{-1}(F_f). $$
The selection of the unique Riemann surface associated with $f\in \mathcal L_\AA$ is made using properties of 
the projection of the zero level of $g_{\RS}$, which is $F_f$:  this set of analytic arcs on the plane has to make $f$ single valued in the complementary domain and also the jump of $f$ across any those arcs must be not identically zero.

The uniqueness of such a Riemann surface implies convergence of the sequence $R'_n/(nR_n)$ (and in consequence, of the sequence of counting measures for the Pad\'e denominators) along the whole $\N$.

\section{The differential equation for Hermite--Pad\'e polynomials} \label{sec:ODE}

The considerations of Section~\ref{s2}  will be extended here to the case of the Hermite--Pad\'e approximants.

\subsection{Proof of Theorem~\ref{main}}\label{proofmain}

From the definition~\eqref{1.1a} and the properties of $\LL$ we know that the remainder $w(z):=R_n(z)$ is a multivalued analytic function in $\C$, with a finite number of branch points. Any branch of this function in the neighborhood of infinity $U=\{z\in \C:\, |z|>R\}$ is a linear combination with constant coefficients of the  functions $Q_{n,k} f_k$, $k=0, 1, \dots, s$. We  will drop the index $n$ in the notations, and write them $Q_k f_k$. It follows that any branch of $w(z):=R(z)$ in $U$ is a solution of the $s$-th order differential equation 
\begin{equation*}
W(z) = \begin{vmatrix} w&w'&w''& \dots & w^{(s)} \\
Q_{0}&Q_{0}'&Q_{0}''& \dots & Q_0^{(s)} \\
Q_{1}f_1&(Q_{1}f_1)'&(Q_{1}f_1)''& \dots & (Q_{1}f_1)^{(s)} \\
\vdots & \vdots & \vdots & & \vdots \\
Q_{s}f_s&(Q_{s}f_s)'&(Q_{s}f_s)''& \dots & (Q_{s}f_s)^{(s)}
\end{vmatrix} \equiv0,\quad z\in{U},
%\label{wronks1}
\end{equation*}
where $f^{(k)}$ denotes the $k$-th derivative of $f$.
Observe that this is an $s$-th  order differential equation with respect to $w$, solved in particular by $w =  Q_j f_j$ and in consequence, by $w =R= R_n$, so that
\begin{equation*}
W(z) = \begin{vmatrix} w&w'&w''& \dots & w^{(s)} \\
Q_{0}&Q_{0}'&Q_{0}''& \dots & Q_0^{(s)} \\
Q_{1}f_1&(Q_{1}f_1)'&(Q_{1}f)''& \dots & (Q_{1}f_1)^{(s)} \\
\vdots & \vdots & \vdots & & \vdots \\
R &R'&R''& \dots & R^{(s)}
\end{vmatrix} \equiv0,\quad z\in{U}.
%\label{wronks2}
\end{equation*}
Expanding these determinant along the first row yields
\begin{equation}
W(z) =  \sum_{j=0}^s W_j(z) w^{(j)}(z) = 0 ,
\label{odeGeneral1}
\end{equation}
where the coefficient  $W_j$, $j=0,1, \dots, s$, is the minor obtained by deleting the first row and the $(j+1)$-th column in the determinantal expression for $W$. 
For instance,
\begin{align*}
W_0(z)& = \begin{vmatrix} 
Q_{0}'&Q_{0}''& \dots & Q_0^{(s)} \\
(Q_{1}f_1)'&(Q_{1}f_1)''& \dots & (Q_{1}f_1)^{(s)} \\
 \vdots & \vdots & & \vdots \\
(Q_{s}f_s)'&(Q_{s}f_s)''& \dots & (Q_{s}f_s)^{(s)}
\end{vmatrix}(z) \\
&=  \begin{vmatrix} 
 Q_{0}'&Q_{0}''& \dots & Q_0^{(s)} \\
 (Q_{1}f_1)'&(Q_{1}f_1)''& \dots & (Q_{1}f_1)^{(s)} \\
  \vdots & \vdots & & \vdots \\
 R'&R''& \dots & R^{(s)}
\end{vmatrix}(z)=\mathcal O\left( z^{-\frac{s(s+3)}{2}}\right), \quad z\to \infty,
\end{align*}
where we have used that $Q_j\in \P_n$ and the asymptotics of $f_k$ and $R$ at infinity. Operating in the same fashion, we conclude that
\begin{equation}
\label{asymptW}
W_j(z)=\mathcal O\left( z^{-\frac{s(s+3)}{2}+j}\right), \quad z\to \infty, \quad j=0, 1, \dots, s.
\end{equation}

The assumption that $f_k\in \mathcal L$ implies that there exist a pair of polynomials, $A_j$ and $B_j$, $A_j$ monic, and $\deg (A_k) - \deg (B_k)\geq 2$, such that
$$
\frac{f_k'}{f_k}(z)=\frac{B_k}{A_k}, 
$$
and thus,
\begin{equation*}
%\label{derivativesF}
A_k^j f_k^{(j)}= B_{k,j}f_k, \quad j=0, 1, 2, \dots,
\end{equation*}
where $B_{k,j}$ are algebraic polynomials. A consequence is that
$$
\frac{A_k^j}{f_k} \left( Q_k f_k \right)^{(j)}\in \P, \quad j=0, 1, \dots, \quad k=0, 1, \dots, s.
$$

Hence, multiplying \eqref{odeGeneral1} by 
$$
\prod_{k=0}^s\frac{A_k^s}{f_k^s}(z),
$$
we obtain the equivalent ODE of the form
\begin{equation*}
\prod_{k=0}^s\frac{A_k^s}{f_k^s}(z)W(z) =  \sum_{k=0}^s \Pi_k(z) w^{(k)}(z) = 0 ,
%\label{odeGeneral2}
\end{equation*}
where all coefficients $\Pi_k$ are polynomials.  By \eqref{asymptW}, their degrees are jointly uniformly bounded.

This proves the theorem.

\subsection{Proof of Theorem ~\ref{t1}}\label{s3s1}

Using the arguments of the proof of Theorem~\ref{main} we conclude that
\begin{equation}
W(z)=\begin{vmatrix} w&w'&w''&w'''\\
Q_{0}&Q_{0}'&Q_{0}''&Q_{0}'''\\
Q_{1}f&(Q_{1}f)'&(Q_{1}f)''&(Q_{1}f)'''\\
Q_{2}f^2&(Q_{2}f^2)'&(Q_{2}f^2)''&(Q_{2}f^2)'''
\end{vmatrix} \equiv0,\quad z\in U=\{z\in \C:\, |z|>R\} ,
\label{3.1}
\end{equation}
where we again omit the subindex $n$. From the assumption that $f\in\LL_\AA$, $\AA = \{a_1,\dots,a_p\}$, $p\geq 2$, it follows that this identity can be extended to $\C\setminus \mathcal A$. Furthermore, there exist algebraic polynomials $B$, $C$ and $D$ with complex coefficients, such that
\begin{equation}
Af'=Bf,\quad A^2f''=Cf,\quad A^3f'''=Df,  \quad A(z)=\prod\limits_{j=1}^p(z-a_j).
\label{3.2}
\end{equation}
As in the proof of Theorem~\ref{main}, after multiplying $W$ by $A^6/f^3$ we get a differential equation with polynomial coefficients. Actually, it is sufficient to multiply $W$ by  $A^5/f^3$.  Indeed, derivatives of order three are all located in the last  column of the determinant in \eqref{3.1}, and therefore, each coefficient in first  row of decomposition contains only one such term. 

Thus, we have proved that 
\begin{equation}                                                                                                         % (26)
\frac{A^5}{f^3}W(z)=
\Pi_3(z)w'''+\Pi_2(z)w''+\Pi_1(z)w'+\Pi_0(z)w \equiv 0,\quad z\in\C\setminus \AA,
\label{3.3}
\end{equation}
where $\Pi_j\in\P$. 
The generic case conditions imply that $\mdeg{Q_{j}}=n$, $j=0,1$, and $R(z)=M_n/z^{2n+2}+\dots$, where $M_n\neq0$. It implies that for polynomials  $\Pi_j$ we have $\mdeg{\Pi_j}=5p-9+j$, $j=0,1,2,3$. Thus, those polynomials depend on  $n$  but their degrees are uniformly bounded. 

It is easy to see that for  $\Pi_3$ we have representation
\begin{equation*}
\Pi_3(z)=A^2_p(z)H(z),
%\label{3.5}
\end{equation*}
where  $H(z)=H_n(z)=z^{3p-6}+\dots\in\P_{3p-6}$. Next, using the fact coefficients of equation \eqref{3.3} were by first row decomposition of Wronskian we obtain (with arguments similar to what was done in Section~\ref{sec21}) a representation for $\Pi_2$:
\begin{equation*}
\Pi_2
=-\biggl(\frac{f^3}{A^5}\Pi_3\biggr)'\frac{A^5}{f^3}
=-\biggl(\frac{f^3}{A^3}H\biggr)'\frac{A^5}{f^3}
=A\{3(A'-B)H-AH'\}.
%\label{3.6}
\end{equation*}
Finally, taking into  account the fact that equation~\eqref{3.3} has  a polynomial $Q_0$ of degree  $\mdeg{Q_0}=n$, as well as another solution with the leading term at infinity of the form  $R^{*}_n(z)=1/z^{2n+2}+\dots$,  we arrive at two equations for the leading coefficients of the polynomials $\Pi_1\in\P_{5p-8}$ and $\Pi_0\in\P_{5p-9}$:
\begin{equation*}
\Pi_1(z)=-3(n-1)(n+2)\Pi_1^{*}(z),\qquad
\Pi_0(z)=2(n-1)n(n+1)\Pi_0^{*}(z).
%\label{3.7}
\end{equation*}
Theorem ~\ref{t1} is proved.

\subsection{Proof of Theorem~\ref{coro:casep2}}\label{s3s2}

We turn to considering the case $p=2$, $a_1=-1$, $a_2=1$ and
$f(z)=f(z;\alpha)=\bigl((z-1)/(z+1)\bigr)^\alpha$, $2\alpha\in\C\setminus\ZZ$.

%{\bad
It is directly verified that in this particular case we have $H(z)\equiv1$,
$\Pi_2(z)=6A_2(z)(z-\alpha)=6(z^2-1)(z-\alpha)$. 

Next, we prove that for polynomials.
$\Pi_1=\Pi_1(z;\alpha)$
and $\Pi_0=\Pi_0(z;\alpha)$ the following equalities are valid:
%}
\begin{equation}
\begin{gathered}
\Pi_1(-z;-\alpha)=\Pi_1(z;\alpha),\qquad
-\Pi_0(-z;-\alpha)=\Pi_0(z;\alpha),\\
\Pi_1(z;-\alpha)=\Pi_1(z;\alpha)-3b_1(z)+6(z-\alpha)b_0,\\
A_2(z)\Pi_0(z;-\alpha)=A_2(z)\Pi_0(z;\alpha)+b_2(z)+6(z-\alpha)b_1(z)
+\Pi_1(z;\alpha)b_0,
\end{gathered}
\label{3.8}
\end{equation}
where polynomials $b_0=\const,b_1(z),b_2(z)$ are defined from the following equations
(see ~\eqref{3.2}):
\begin{equation}
\begin{gathered}
A_2(f^2)'=4\alpha f^2=b_0f^2,\qquad
A^2_2(f^2)''=8\alpha(2\alpha-z)f^2=b_1(z)f^2,\\
A_2^3(f^2)'''=8\alpha\{(2\alpha-z)4\alpha-4A_2'(2\alpha-z)-A_2\}f^2=
b_2(z)f^2.
\end{gathered}
\label{3.9}
\end{equation}

We write equation ~\eqref{1.9} in the form
\begin{equation}
A_2^2w'''+\Pi_2w''+\Pi_1w'+\Pi_0w=0,
\label{A.1}
\end{equation}
where $\Pi_2(z)=6(z-\alpha)A_2(z)$, $A_2(z)=(z^2-1)$ and $\Pi_1\in\P_{5p-8}=\P_2,
\Pi_2\in\P_{5p-9}=\P_1$. In particular, polynomial $\Pi_2(z)$ is known. Let
$\Pi_1(z)=k_2z^2+d_2z+c_2$, $\Pi_0(z)=k_1z+c_1$; in the generic case we have $k_j\neq0$, $d_2,c_1,c_2$, functions $Q_{n,0},Q_{n,1}f$, $Q_{n,2}f^2$ and $R_n$ are solutions of this equation, and their monic versions satisfy  
$Q_{n,0}^*(z)=z^n+\dots$, $Q_{n,2}^*(z)=z^n+\dots$, $R_n^*(z)=1/z^{2n+2}+\dots$. Substitute $w=Q_0$  in~\eqref{A.1}; then the left hand side of the equation is a polynomial. Equating its coefficient  at $z^{n+1}$ (the leading coefficient) to zero we obtain 
\begin{equation}
n(n-1)(n+4)+nk_2+k_1=0.
\label{A.2}
\end{equation}
Similarly, substituting $w=R_n^*$ and equating to zero the coefficient at $1/z^{2n+1}$ we get another equation for
$k_1,k_2$:
\begin{equation}
-4(n+1)(2n+3)(n-1)-2(n+1)k_2+k_1=0.
\label{A.3}
\end{equation}
Solving system ~\eqref{A.2} and~\eqref{A.3} for $k_1$ and $k_2$ we obtain
\begin{equation*}
k_1=2(n-1)n(n+1)=2(n^2-1)n,\qquad k_2=-3(n-1)(n+2).
%\label{A.4}
\end{equation*}

From normalization  $f(\infty;\alpha)=1$ and identity $f(z;-\alpha)=1/f(z;\alpha)$, by 
dividing both sides of the definition~\eqref{1.1N} by $f^2$ we come to
$Q^{(\alpha,-\alpha)}_{n,2}=Q^{(-\alpha,\alpha)}_{n,0}$.
Therefore,
the polynomial $Q_{n,2}$ satisfies a differential equation which is obtained from \eqref{A.1},
by replacing $\alpha$ by $-\alpha$. The modified equation is
\begin{equation}
A_2^2\myt w'''+\myt\Pi_2\myt w''+\myt\Pi_1\myt w'+\myt\Pi_0\myt w=0,
\label{A.5}
\end{equation}
where $\myt\Pi_2(z)=6(z+\alpha)A_2(z)$, $\myt\Pi_j(z)=\Pi_j(z;-\alpha)$, $j=1,2$.

Next, substitute the solution $w=Q_{n,2}f_2$ where $f_2=f^2\in\LL_{\{-1,1\}}$ in~\eqref{A.1} and regroup terms so that equation takes form~\eqref{A.5} with solution $\myt w=Q_{n,2}$. This yields the following relation for coefficients $\Pi_j$ and $\myt\Pi_j$ of equations~\eqref{A.1} and~\eqref{A.5}:
\begin{gather}
3A_2B_0+\Pi_2=\myt\Pi_2,\notag\\
3B_1+2\Pi_2\frac{B_0}{A_2}+\Pi_1=\myt\Pi_1,
\label{A.6}\\
\frac{B_2}{A_2}+\Pi_2\frac{B_1}{A_2^2}+\Pi_1\frac{B_0}{A_2}+\Pi_0=\myt\Pi_0,
\label{A.7}
\end{gather}
where $\Pi_2(z)=6(z-\alpha)A_2(z)$,
\begin{equation}
\gathered
A_2f_2'=4\alpha f_2=B_0f_2,\qquad
A_2^2f_2''=8\alpha(2\alpha-z)f_2=B_1f_2,\\
A_2^3f_2'''=f_28\alpha\{(2\alpha-z)4\alpha-4B_2A_2'(2\alpha-z)-A_2\}=
B_2f_2.
\endgathered
\label{A.8}
\end{equation}
It follows from~\eqref{A.6} and~\eqref{A.8} 
\begin{equation}
24\alpha z+\Pi_1(z)=\myt\Pi_1(z)\equiv\Pi_1(z;-\alpha).
\label{A.9}
\end{equation}
Using the symmetry $f(-z;-\alpha)=f(z;\alpha)$ and the differential equations~\eqref{A.1} and~\eqref{A.5}, we get
\begin{equation}
\Pi_1(-z;-\alpha)=\Pi_1(z;\alpha),\qquad
\Pi_0(-z;-\alpha)=-\Pi_0(z;\alpha).
\label{A.10}
\end{equation}
Therefore, from ~\eqref{A.9} and~\eqref{A.10}  we obtain
\begin{equation*}
\begin{aligned}
\Pi_1(z;\alpha)&=k_2z^2-12\alpha z+c_2,\quad\text{where}\quad k_2=-3(n-1)(n+2),\\
\Pi_0(z;\alpha)&=k_1z+c_1,\quad\text{where}\quad  k_1=2(n-1)n(n+1).
\end{aligned}
%\label{A.11}
\end{equation*}
Taking the limit as $z\to\infty$ in~\eqref{A.7} and with the equation $c_1(-\alpha)=-c_1(\alpha)$ (following from~\eqref{A.10}) we find that
\begin{equation}
c_1(\alpha)=2\alpha(3n(n+1)-8).
\label{A.12}
\end{equation}
Similarly, evaluating~\eqref{A.7} at $z=0$ and combining it with~\eqref{A.12} it follows that
\begin{equation*}
c_2(\alpha)=3n(n+1)+8\alpha^2-10.
%\label{A.13}
\end{equation*}

Finally, from~\eqref{3.8}--\eqref{3.9} one can find polynomials $\Pi_1$ and $\Pi_0$ explicitly:
\begin{align}
\Pi_1(z;\alpha)&=-3(n-1)(n+2)z^2-12\alpha z+\bigl[3n(n+1)+8\alpha^2-10\bigr],
\notag\\
\Pi_0(z;\alpha)&=2n(n^2-1)z+2\alpha\bigl[3n(n+1)-8\bigr].
\notag
\end{align}
This concludes the proof.

\section{Asymptotics of Hermite--Pad\'e polynomials}\label{s4}

\subsection{Proof of Theorem~\ref{t2}} \label{subsecThmt2}

Denote $E=[-1,1]$. 
Definition~\eqref{1.1N} (with $s=2$) yields immediately the following orthogonality relation
\begin{equation}
\label{3}
\oint_E(Q_{n,1}f+Q_{n,2}f^2)(\zeta)q(\zeta)\,d\zeta=0,
\end{equation}
where $q\in\P_{2n}$ is arbitrary, and we integrate along an arbitrary closed contour encircling and sufficiently close  to $E$.
Using the construction of Pad\'e polynomials (see e.g.~\eqref{Pade}) we find
$P_{n,0}, P_{n,1}\in\P_n$, $P_{n,1}\not\equiv0$, such that
\begin{equation}
\label{4}
(P_{n,0}+P_{n,1}f)(z)=\mathcal O\left(\frac1{z^{n+1}}\right),\quad z\to\infty.
\end{equation}
From~\eqref{4} it follows immediately that
\begin{equation}
\label{5}
\oint_E P_{n,1}(\zeta)f(\zeta)p(\zeta)\,d\zeta=0,\quad\forall p\in\P_{n-1}.
\end{equation}
Since $f(x)>0$ when $x>1$, we have
\begin{equation}
\label{6}
f^{\pm}(x)=e^{\pm i\pi\alpha} f_0(x),\quad
x\in(-1,1),
\end{equation}
where $f^+$ (resp., $f^-$) are the boundary value of $f$ on $(-1,1)$
from the upper (resp., lower) halfplane, and $ f_0$ was defined in \eqref{falphNalt}. 
In consequence,
\begin{equation}
\label{defDeltaF}
\Delta f(x):=(f^{+}-f^{-})(x)=2i\sin(\alpha)  f_0(x)=2i\sin(\alpha)\left(\dfrac{1-x}{1+x}\right)^\alpha, \quad x\in (-1,1).
\end{equation}
Taking into account that $|\alpha|\in (0, 1/2)$, \eqref{6} can be rewritten as
$$
\int_{-1}^1 P_{n,1}(x) p(x) \Delta f(x) \,dx=0,\quad\forall p\in\P_{n-1},
$$
and we conclude that polynomials $P_{n,1}=P_n$ coincide (up to normalization) with the Jacobi polynomials $P_n^{(\alpha,-\alpha)}$ orthogonal on the segment $E$ with respect to the positive weight function $f_0$:
\begin{equation*}
%\label{7}
\int_E P_n(x)p(x)f_0(x)\,dx=0,\quad p\in\P_{n-1}.
\end{equation*}
In order to simplify notation, let us define also
\begin{equation*}
%\label{8}
\myt{f}(z):=2\cos(\alpha\pi) f_0(z),\quad
z\in\Omega:=\C\setminus{F},\quad F:=\overline \R\setminus (-1, 1),  
\end{equation*}
so that
$$
\myt{f}(x)=(f^{+}+ f^{-})(x), \quad x\in (-1,1).
$$
The integrability of $f_0^2$ at the end points of $E$ plus the boundary conditions~\eqref{6} allow us to rewrite  the relation~\eqref{3} as
\begin{equation}
\label{9}
\int_E(Q_{n,1}+Q_{n,2}\myt{f})(x)q(x)f_0(x)\,dx=0,\quad \forall q\in\P_{2n}.
\end{equation}
Since $Q_{n,k}$ have real coefficients and both functions $\myt{f}$ and $f_0$ are positive in
the interval $(-1,1)$, it follows  from~\eqref{9} that the
form 
\begin{equation}
\label{defRhoN}
\rho_n(z):=(Q_{n,1}+Q_{n,2}\myt{f})(z)
\end{equation}
has at least $2n+1$ zeros in the interval $(-1,1)$, that we denote by
$x_{n,j}$. Let 
$$
\omega_{2n+1}(z):=\prod\limits_{j=1}^{2n+1}(z-x_{n,j}).
$$ 
Since  for each  $p\in\P_{n-1}$ function
$\rho_n p/\omega_{2n+1}$ is holomorphic
 in the domain $\Omega$,  by the Cauchy formula we obtain
\begin{align}
\frac{Q_{n,1}+Q_{n,2}\myt{f}}{\omega_{2n+1}}(z)
&=\frac1{2\pi i}\oint_\gamma
\frac{(Q_{n,1}+Q_{n,2}\myt{f})(\zeta)\,d\zeta}{\omega_{2n+1}(\zeta)(\zeta-z)} =\frac1{2\pi i}\oint_\gamma
\frac{(Q_{n,2}\myt{f})(\zeta)\,d\zeta}{\omega_{2n+1}(\zeta)(\zeta-z)},
\label{10}
\end{align}
where $\gamma$ is an arbitrary simple analytic contour that is contained
in the domain $\Omega$ and that contains all the points $x_{n,j}$,
$j=1,\dots,2n+1$, and $z$ inside. Also
\begin{equation}
\label{11}
0=\frac1{2\pi i}\oint_\gamma
\frac{(Q_{n,1}+Q_{n,2}\myt{f})(\zeta)p(\zeta)\,d\zeta}{\omega_{2n+1}(\zeta)}
=\frac1{2\pi i}\oint_\gamma
\frac{(Q_{n,2}\myt{f})(\zeta)p(\zeta)\,d\zeta}{\omega_{2n+1}(\zeta)},
\end{equation}
 It is easy to see that we can
transform the contour $\gamma$ in such a way that the
relation~\eqref{11} takes the form
\begin{equation}
\label{12}
0=\int_F\frac{Q_{n,2}(y)p(y)\Delta\myt{f}(y)\,dy}{\omega_{2n+1}(y)},\quad \Delta\myt{f}(y)= \myt{f}^{+}(y) -\myt{f}^{-}(y), \quad
p\in\P_{n-1}.
\end{equation}
Since
\begin{equation*}
%\label{13}
\myt{f}^{\pm}(y)=e^{\mp\alpha\pi i}2\cos(\alpha\pi) \left|\frac{1-y}{1+y}\right|^\alpha,
\quad y\in F^0:=[-\infty,-1)\cup(1,\infty],
\end{equation*}
we obtain that $\Delta\myt{f}(y)=-4i\sin(2\alpha\pi)|(y-1)/(y+1)|^\alpha$,
$y\in F^0$. Thus the relation~\eqref{12} may be written in an equivalent form
as
\begin{equation}
\label{14}
0=\int_F Q_{n,2}(y)p(y)h_{n,\alpha}(y)\,dy,\quad p\in\P_{n-1},
\end{equation}
where the weight function
$h_{n,\alpha}(y)=|(y-1)/(y+1)|^\alpha/\omega_{2n+1}(y)$. Since
$\omega_{2n+1}$ is of degree $2n+1$, function $h_{n,\alpha}$ is
negative on $(-\infty,-1)$ and positive on $(1,\infty)$, and it
follows immediately from~\eqref{14} that real polynomial
$Q_{n,2}$ has at least $n-1$ zeros on $F^0$. 

If $\mdeg{Q_{n,2}}=n-1$, it shows that 
all  its zeros  are contained in the $F^0$. Let
$\mdeg{Q_{n,2}}=n$. Since $Q_{n,2}$ is a real polynomial with at least $n-1$ real zeros, we conclude that all zeros of $Q_{n,2}$ are real, and at most one of them in $E$. Assuming the existence of one zero in $E$ leads into contradiction with the  orthogonality
relation~\eqref{14}. Thus, once again we conclude that all zeros of $Q_{n,2}$ are in $F^0$.

Since same arguments can be applied to  the function $(Q_{n,1}/\myt{f}+Q_{n,2})p/\omega_{2n+1}$, it immediately follows that we arrive at the same conclusion about the zeros of  the polynomial $Q_{n,1}$. Finally, we can
divide both sides of~\eqref{HPfor2and2} by $f^{2}$ to obtain that
$Q_{n,0}(z;f)=Q_{n,2}(z;f^{-2})$.

Hence, we have proved that all zeros of
polynomials $Q_{n,0}$, $Q_{n,1}$, and $Q_{n,2}$ are contained in $F^0$.

Let us remark also that by~\eqref{10}, \eqref{11} and~\eqref{12},
\begin{align}
(Q_{n,1}+Q_{n,2}\myt{f})(z)
&=\omega_{2n+1}(z)\frac1{2\pi i}\oint_\gamma\frac{Q_{n,2}(\zeta)\myt{f}(\zeta)}
{\omega_{2n+1}(\zeta)(\zeta-z)}d\zeta\notag\\
&=\frac{\omega_{2n+1}(z)}{Q_{n,2}(z)}\frac1{2\pi i}\oint_\gamma\frac{Q_{n,2}^2(\zeta)\myt{f}(\zeta)}
{\omega_{2n+1}(\zeta)(\zeta-z)}d\zeta\notag\\
&=\frac{\omega_{2n+1}(z)}{Q_{n,2}(z)}
\int_F\frac{Q_{n,2}^2(y)\Delta\myt{f}(y)}{\omega_{2n+1}(y)(y-z)}dy,\quad
z\in\Omega.
\label{15}
\end{align}

Now we turn to the asymptotics. 
Divide the equation~\eqref{1.9} through by $n^3 w$ and rewrite it in terms of $h_n=w'/(n w)$:
\begin{equation*}
\begin{split}
& (z^2 - 1)^2\left(\frac{1}{n^3}h_n'' + \frac{3}{n} h_n h_n'+h_n^3\right)  + \frac{6}{n}(z^2 - 1)(z -\alpha )\left(\frac{1}{n}h_n'+h_n^2\right) \\  &
- \left[\frac{3(n-1)(n+2)}{n^2}z^2
+\frac{12\alpha z}{n^2}-\frac{3n(n+1)+8\alpha^2 -10}{n^2}\right]h_n \\
& +2\frac{n(n^2 -1)z+\alpha (3n(n+1)-8)}{n^3}=0.
\end{split}
\end{equation*}
Since $w=Q_{n,0}$ is a solution of~\eqref{1.9}, the corresponding $h_n$ takes the form
$$
h_n(z)= C^{\mu_n}(z) = \int \frac{d\mu_n(t)}{z-t} ,
$$
where $\mu_n=\mu(Q_{n,0})$ is the normalized zero-counting measure of $ Q_{n,0}$.

Now we take limit along any convergent subsequence\footnote{At this point we
should assume it convergent in the extended complex plane $\C\cup
\{\infty\}$.}, with $\mu_n\stackrel{*}{\to} \nu$, so that $h_n \to h$, $z\notin
\R$, and get that $h$ satisfies the cubic equation
\begin{equation} \label{cubic}
(z^2 - 1)^2 h^3(z)  - 3(z^2 -1) h(z)  +2z=0.
\end{equation}
We apply the standard Cardano formula to find its solution. This yields
the following result:
\begin{lemma}\label{lemma1}
The general solution of the equation
\begin{equation} \label{cubicGeneral}
(z^2 - 1)^2 y^3(z)  - 3(z^2 -1) y(z)  +2z=0
\end{equation}
for $z\in \C\setminus\{-1,1 \}$ is
\begin{equation} \label{generalFormula}
y(z)= \frac{1}{z+1}\, Y(z) +   \frac{1}{z-1}\, \frac{1}{Y(z)},
\end{equation}
where
\begin{equation} \label{defF}
Y(z) = \left(\frac{1+z}{1-z}\right)^{1/3}.
\end{equation}
In a neighborhood of infinity there are three holomorphic and linearly independent solutions $y_j$, $j=1, 2, 3$, that can be enumerated in such a way that they satisfy
\begin{equation}\label{atinfinity}
y_1(z) + y_2(z)+ y_3(z) =0, \quad 
\begin{cases}
y_1(z)=\dfrac{1}{z} +\mathcal O\left(\dfrac{1}{z^2}\right), &   \\[3mm]
y_2(z)=\dfrac{1}{z} +\mathcal O\left(\dfrac{1}{z^2}\right), & z\to \infty. \\[3mm]
y_3(z)=-\dfrac{2}{z} +\mathcal O\left(\dfrac{1}{z^2}\right), &  
\end{cases}
\end{equation}

\end{lemma}
\begin{proof}
A direct substitution of the right hand side in \eqref{generalFormula} into
\eqref{cubicGeneral} yields that $y(z)$ given by \eqref{generalFormula}  is
indeed a solution of this cubic equation, regardless the branch of $Y$ considered.

Furthermore, let $y_1(z)$ and $y_2(z)$, $z\neq \pm 1$, be two values of
\eqref{generalFormula} corresponding to two different selections of the branch of $Y$. Then, without loss of generality, we can take the value of $Y(z)$ in such a way that
$$
y_1(z)= \frac{1}{z+1}\, Y(z) +   \frac{1}{z-1}\, \frac{1}{Y(z)}, \quad y_2(z)= \frac{e^{2\pi i/3}}{z+1}\, Y(z) +   \frac{e^{-2\pi i/3}}{z-1}\, \frac{1}{Y(z)}.
$$
Assuming that $y_1(z)=y_2(z)$, straightforward calculations yield us into a contradiction. Finally, relations \eqref{atinfinity} are obtained by totally standard arguments. The lemma is proved.
\end{proof}

Observe that the remainder function $R_n$, defined in \eqref{HPfor2and2}, is is holomorphic  in
$D:=\myo\C\setminus{E}$ and is a solution of the differential
equation~\eqref{1.9}; in consequence, its analytic jump $w_n=\Delta R_n$
on $E^0=(-1,1)$ is also a solution  of the
same differential equation.  But
$$
w_n(x)=\rho_n(x) \Delta f(x), \quad x\in (-1,1),
$$  
where $\rho_n$ was defined in \eqref{defRhoN} and $\Delta f$ is given by \eqref{defDeltaF}. 
Standard arguments show that for $z\notin F$,
$$
\frac1n\biggl\{\int_F\frac{Q_{n,2}^2(y)\Delta\myt{f}(y)\,dy}{\omega_{2n+1}(y)(y-z)^2}
\biggr\}
\biggl\{\int_F\frac{Q_{n,2}^2(y)\Delta\myt{f}(y)\,dy}{\omega_{2n+1}(y)(y-z)}
\biggr\}^{-1}\to0,
\quad n\to\infty,
$$
and we conclude from~\eqref{15} that
\begin{align}
\frac1n\frac{w_n'}{w_n}
&=\frac1n\frac{\rho_n'}{\rho_n}+\frac1n\frac{(\Delta f)'}{\Delta f}
=\frac1n\frac{\omega_{2n+1}'}{\omega_{2n+1}}(z)-\frac1n\frac{Q_{n,2}'}{Q_{n,2}}(z)+o(1). \notag
\end{align}
Using the result of Theorem~\ref{t2}, we conclude that
$$
\lim_{n} \frac1n\frac{w_n'}{w_n}= 2C^{\lambda}(z)- C^{\nu}(z), \quad z\in \C\setminus \R.
$$

From these considerations it follows that $C^\nu$ and $2C^{\lambda} - C^{\nu} $ are two, clearly linearly independent solutions of the algebraic equation \eqref{cubicGeneral}. We already know that
$$
C^\nu(z)=y_1(z)=\frac{1}{z} + \mathcal O\left(\frac{1}{z^2} \right), \quad z \to \infty.
$$
Taking into account \eqref{atinfinity} we conclude that
$$
(2C^{\lambda} - C^{\nu} )(z)=y_2(z)= \frac{1}{z} + \mathcal O\left(\frac{1}{z^2} \right) , \quad C^\lambda(z)=-\frac{1}{2}y_3(z)=\frac{1}{z} + \mathcal O\left(\frac{1}{z^2} \right),
$$
which proves in particular that $\lambda$ is a probability measure supported on $E$.

Since we know the support of both $\lambda$ and $\nu$, we can recover 
their densities using the Sokhotskii-Plemelj formulas,
\begin{align} \label{sokhotski}
\nu'(z)& =-\frac{1}{2\pi i}\left(y_{1+}(x)-y_{1-}(x)\right) ,
\quad z\in\supp(\nu)=F=\overline \R\setminus(-1,1), \\
\lambda'(z)& =\frac{1}{4\pi i}\left(y_{3+}(x)-y_{3-}(x)\right)  ,
\quad z\in\supp(\nu)=E=[-1,1]. \label{sokhotski1}
\end{align}

Let $Y$ now denotes the holomorphic branch of \eqref{defF}  in $\C\setminus
F$, determined by $Y(0)=1$. With this convention,
\begin{equation*}
%\label{limitY}
\lim_{x\to \pm \infty} Y(i x)=e^{\pm \pi i/3}.
\end{equation*}

Since $y_1(x)=C^\nu(x)$ and $\supp(\nu)=\overline \R\setminus (-1,1)$, it follows that $y_1(x)\in \R$ for $x\in (-1,1)$. This allows us to single out the expression for $y_1$. Expanding at infinity we obtain the expression for all three solutions in $\C\setminus \R$:
\begin{align}
\label{expressionH}
y_1(z)& = C^\nu(z)=\frac{1}{z+1}Y(z)+\frac{1}{z-1}\frac{1}{Y(z)},
  \\
y_2(z)& = (2C^{\lambda} - C^{\nu} )(z)=-y_1(z)-y_3(z),
  \\ \label{expressionH3}
y_3(z)& = -2C^\lambda(z)=\begin{cases} 
\dfrac{e^{2\pi i/3}}{z+1}Y(z)+\dfrac{e^{-2\pi i/3}}{z-1}\dfrac{1}{Y(z)}, & \Im z > 0, \\[3mm]
\dfrac{e^{-2\pi i/3}}{z+1}Y(z)+\dfrac{e^{2\pi i/3}}{z-1}\dfrac{1}{Y(z)}, & \Im z < 0.
\end{cases}
\end{align}
Let $Y_+$ (reps., $Y_-$) denote the boundary values of the selected branch of $Y$ on $(1,+ \infty)$ from the upper (resp., lower) half plane\footnote{Calculations for $(-\infty, -1)$ are similar.}. Then
$$
Y_\pm (x) =e^{\pm \pi i/3} \sqrt[3]{\frac{x+1}{x-1}}, \quad x>1,
$$
where we take the positive values of $\sqrt[3]{\cdot}$, so that
\begin{align*}
\Delta Y(x)& =Y_+(x)- Y_-(x)= i\sqrt{3} \sqrt[3]{\frac{x+1}{x-1}},
\\ 
\Delta \frac{1}{Y}(x)& =\frac{1}{Y_+(x)}- \frac{1}{Y_-(x)}= -i\sqrt{3} \sqrt[3]{\frac{x-1}{x+1}}.
\end{align*}
Using \eqref{sokhotski} and \eqref{expressionH} we see that
\begin{align*}
\nu'(z)  &=-\frac{1}{2\pi i}\left( \frac{1}{x+1}\Delta Y(x)+\frac{1}{x-1}\Delta \frac{1}{Y}(x) \right) \\
&= -\frac{\sqrt{3}}{2\pi  }\left( \frac{1}{x+1} \sqrt[3]{\frac{x+1}{x-1}}-\frac{1}{x-1}  \sqrt[3]{\frac{x-1}{x+1}} \right),
\end{align*}
which establishes \eqref{1.91}. 

On the other hand, for $x\in (-1,1)$, by \eqref{sokhotski1} and \eqref{expressionH3},
\begin{align*}
\lambda'(z)  &=\frac{\sin (2\pi /3)}{2\pi}\left( \frac{1}{x+1} Y(x)-\frac{1}{x-1} \frac{1}{Y(x)} \right) \\
&= \frac{\sqrt{3}}{4\pi  }\left( \frac{1}{x+1} \sqrt[3]{\frac{1+x}{1-x}}+\frac{1}{1-x}  \sqrt[3]{\frac{1-x}{1+x}} \right),
\end{align*}
where we again take the positive values of $\sqrt[3]{\cdot}$, and \eqref{er2} follows.

\subsection{Proof of Theorem~\ref{t4}}\label{s6}

By~\eqref{15},
\begin{equation}
\label{e22}
\frac{Q_{n,1}}{Q_{n,2}}(z)+\myt{f}(z)= \frac{\rho_n}{Q_{n,2}}(z)
=\frac{\omega_{2n+1}(z)}{Q^2_{n,2}(z)}
\int_F\frac{Q^2_{n,2}(y)\Delta\myt{f}(y)\,dy}{\omega_{2n+1}(y)(y-z)},
\end{equation}
with $\rho_n$ defined in \eqref{defRhoN}. From Theorem~\ref{t2} it follows that
\begin{equation}
\label{e23}
\frac1n\frac{\rho_n'}{\rho_n}-\frac1n\frac{Q_{n,2}'}{Q_{n,2}}\to
2(C^{\lambda} - C^{\nu} )(z), \quad z\in \C\setminus \R.
\end{equation}
In the previous section we have established that
$$
y_1(z)=C^\nu(z), \quad   y_2(z)= (2C^{\lambda} - C^{\nu} )(z), \quad y_3(z)=-2C^{\lambda} (z)
$$
are three independent holomorphic solutions of the algebraic equation \eqref{cubicGeneral} in $\C\setminus \R$. With this notation we conclude from ~\eqref{e23} that
\begin{equation*}
\frac1n\log\frac{\rho_n}{Q_{n,2}}(z)\to \int^z\{y_2(\zeta)-y_1(\zeta)\}\,d\zeta,
\quad z\not\in\R,
\end{equation*}
or equivalently,  
\begin{equation}
\label{e25}
\biggl|\frac{\rho_n}{Q_{n,2}}(z)\biggr|^{1/n}\to
\exp\biggl\{ \Re\int^z(y_2-y_1)(\zeta)\,d\zeta \biggr\}
\end{equation}
From~\eqref{e22} and~\eqref{e25} it follows  that \eqref{4.4} is established if we  prove that   
\begin{equation}
\label{trajectories}
\Re\int^z(y_2-y_1)(\zeta)\,d\zeta<0, \quad z\not\in\R.
\end{equation}

Let us consider the three sheeted Riemann surface $\RS_3$ of genus $0$ given by the equation
$w^3=(z-1)/(z+1)$. It can be realized as shown on Figure~\ref{fig:RS}.

%%%%%%%%%%%%%%%%%%%%%%%%%%%%%%%%%%%%%%%%%%%%%%%%%%%%%%%%%%
\begin{figure}[htb]
\centering \begin{overpic}[scale=0.9]%
{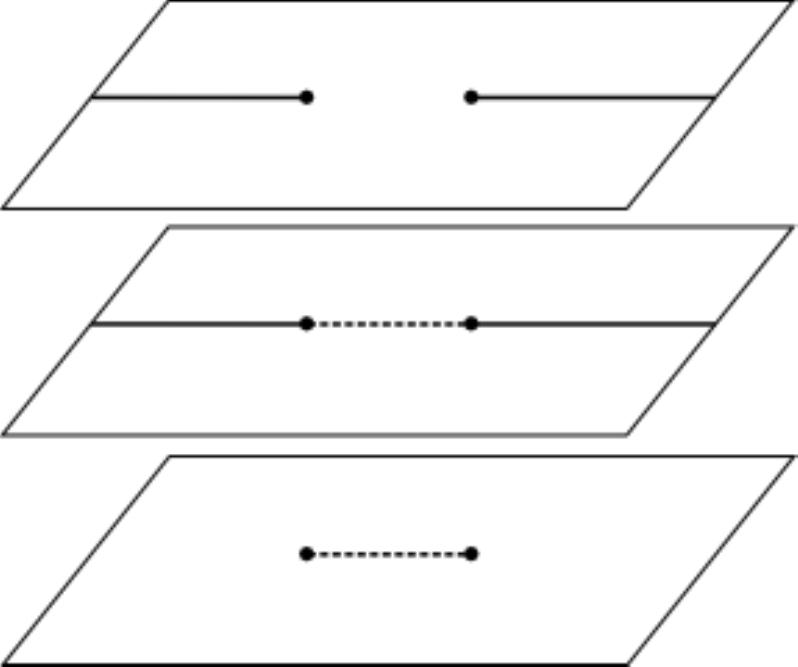}%
 \put(35,67){$-1 $}
 \put(57,67){$1 $}
  \put(35,9){$-1 $}
 \put(57,9){$1 $}
 \put(85,78){$\mathcal R_3^{(1)} $}
 \put(85,49.5){$\mathcal R_3^{(2)} $}
 \put(85,21){$\mathcal R_3^{(3)} $}
\end{overpic}
\caption{Sheet structure of the Riemann surface $\RS_3$.}
\label{fig:RS}
\end{figure}
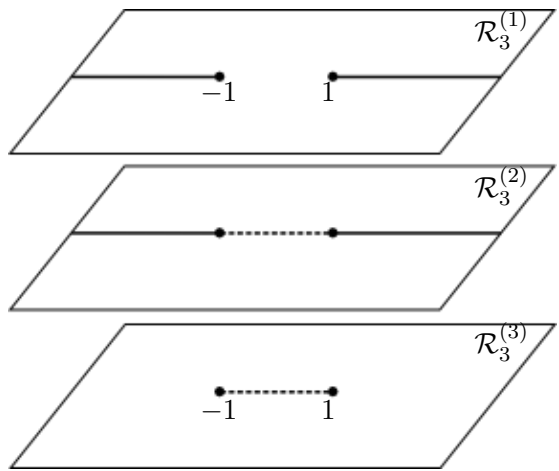
%%%%%%%%%%%%%%%%%%%%%%%%%%%%%%%%%%%%
 
A general approach for establishing \eqref{trajectories} is through the analysis of the global structure of the critical trajectories of the  quadratic differential $\mathfrak F(z)dz^2$ on $\RS_3$, where
$$
\mathfrak F(z) =\begin{cases}
-\left( y_2-y_3\right)^2(z), & \text{if } z \in \RS_3^{(1)}, \\
-\left( y_1-y_3\right)^2(z), & \text{if } z \in \RS_3^{(2)}, \\
-\left( y_1-y_2\right)^2(z), & \text{if } z \in \RS_3^{(3)}
\end{cases}
$$
is meromorphic on $\RS_3$. Here again we use the notation  $ \pi :\, \RS_3\to\overline\C$ for the canonical projection of $\RS_3$, and $  \pi^{-1}(z) = \{z^{(1)}, z^{(2)}, z^{(3)}\}$  for $z\in \overline \C$, with $z^{(j)}\in \RS_3^{(j)}$.

However, in the particular case under consideration we can use a result of J.~Nuttall; following Nuttall's approach, it is sufficient to show that the three functions $y_j(z)$
give the canonical (in Nuttall's sense, see~\cite[sec.
4.3.4]{MR769985},~\cite[\S 6]{MR3137137}) partition of $\RS_3$ into three
sheets in the following way:
\begin{equation}
\label{e21}
\Re \int^z y_3(z)\,dz<\Re \int^z y_2(z)\,dz<\Re \int^z y_1(z)\,dz,\quad z\not\in\R.
\end{equation}
Alternatively, we need an abelian integral $\phi$ on $\RS_3$ such that
\begin{itemize}
\item $\Re (\phi)$ is single valued on $\RS_3$ and harmonic on $\RS_3 \setminus \cup_{j=1}^3 \infty^{(j)}$;
\item $\exp(\phi(z))$ is meromorphic on $\RS_3$, whose divisor on $\RS_3$ is $2 \infty^{(3)}-  \infty^{(1)}-   \infty^{(2)}$.
\end{itemize} 
With such a  $\phi$,  Nuttall's canonical partition~\eqref{e21}  of
$\RS_3$ into three sheets is given by 
\begin{equation*}
%\label{e28}
\Re\phi(z^{(3)})<\Re\phi(z^{(2)})<\Re\phi(z^{(1)}).
\end{equation*}

The key fact we can exploit is that the algebraic curve defined by \eqref{cubicGeneral} is independent of $\alpha$, so we can set  $\alpha=1/3$ and use the Nuttall's example~\cite[Section 4.3.4]{MR769985}, where he
considered the Riemann surface $\RS_3'$ of the equation $zw^3=z-1$, along with the Hermite--Pad\'e approximants to  the system  $\bm f=(1,f,f^2)$ with $f(z)=w$. Nuttall showed that $\exp(\phi(z))=z(1-w)^3$, as well as proved that\footnote{The numeration of the sheets of $\RS_3'$ in \cite{MR769985} is slightly different.}
\begin{equation}
\label{e30}
\Re\phi(z^{(2)})-\Re\phi(z^{(1)})<0, \quad z\notin\R. 
\end{equation}
 
Clearly, the Riemann surfaces $\RS_3$ and $\RS_3'$ are isomorphic in such a way that  \eqref{e30} implies \eqref{trajectories}.

Finally, recall that  identity $f(z;-\alpha)=1/f(z;\alpha)$ implies, by 
dividing both sides of the definition~\eqref{1.1N} by $f^2$, that 
$Q^{(\alpha,-\alpha)}_{n,2}=Q^{(-\alpha,\alpha)}_{n,0}$. Hence, \eqref{4.4} applied to $f(z;-\alpha)$ yields 
$$
\frac{Q_{n,1}}{Q_{n,0}}(z)\longrightarrow
-2\cos(\alpha\pi) / f_0(z),
$$
which allows to conclude \eqref{4.5}. The theorem is proved.

\def\cprime{$'$}

\obeylines
\texttt{
A. Mart\'{\i}nez-Finkelshtein (andrei@ual.es)
Department of Mathematics
University of Almer\'{\i}a, Spain, and
Instituto Carlos I de F\'{\i}sica Te\'{o}rica y Computacional
Granada University, Spain
\medskip
Evguenii A.~Rakhmanov (rakhmano@mail.usf.edu)
Department of Mathematics,
University of South Florida, USA
\medskip
Sergey P.~Suetin (suetin@mi.ras.ru)
Steklov Institute of Mathematics of the 
Russian Academy of Sciences, Russia
}

\end{document}